\documentclass[11pt]{article}

\usepackage[utf8]{inputenc}
\usepackage{leftidx}
\usepackage{amsmath}
\usepackage{amscd}
\usepackage{amssymb}
\usepackage{rotating}
\usepackage{amsfonts,bm}
\usepackage{amsthm}
\usepackage{verbatim}
\usepackage{stmaryrd}
\usepackage{mathrsfs}
\usepackage{color}
\usepackage[top=1in, bottom=1in, left=1.25in, right=1.25in]{geometry}

\usepackage{titlesec}
       \titleformat{\chapter}[display]
             {\normalfont\Large\bfseries}{\thechapter}{11pt}{\Large}
       \titleformat{\section}
             {\normalfont\large\bfseries}{\thesection}{11pt}{\large}
       \titlespacing*{\chapter}{0pt}{0pt}{15pt} 
       \titlespacing*{\section}{0pt}{3.5ex plus 1ex minus .2ex}{2.3ex plus .2ex}

    \usepackage[titletoc]{appendix}

\input xy
\xyoption{all}

\allowdisplaybreaks[1]

\newcommand{\pqed}{\hfill\qedsymbol\\}

\newcommand{\Hom}{\mathrm{Hom}}
\newcommand{\id}{\mathrm{id}}

\newcommand{\Spec}{\mathrm{Spec}}
\newcommand{\coker}{\mathrm{coker}}
\DeclareMathOperator{\Ext}{Ext}

\DeclareMathOperator{\Tor}{Tor}

\newcommand{\C}{\mathbb{C}}
\newcommand{\Q}{\mathbb{Q}}
\newcommand{\Z}{\mathbb{Z}}
\newcommand{\sC}{\mathcal{C}}

\newcommand{\Ab}{\operatorname{Ab}}
\newcommand{\DM}{\operatorname{DM}}
\newcommand{\DA}{\operatorname{DA}}
\newcommand{\Ker}{\operatorname{Ker}}
\newcommand{\Coker}{\operatorname{Coker}}

\newcommand{\gm}{{\operatorname{gm}}}
\newcommand{\et}{{\operatorname{\acute{e}t}}}

\usepackage[colorlinks,linkcolor=blue,citecolor=blue,urlcolor=red]{hyperref}

\newtheorem{theorem}{Theorem}[section]
\newtheorem{theorem/definition}{Theorem/Definition}[section]

\newtheorem{proposition}[theorem]{Proposition}
\newtheorem{lemma}[theorem]{Lemma}

\newtheorem{corollary}[theorem]{Corollary}

\newtheorem{conjecture}[theorem]{Conjecture}

\theoremstyle{remark}
\newtheorem{remark}[theorem]{Remark}

\theoremstyle{definition}
 
\newtheorem{definition}[theorem]{Definition}

\newtheorem{rk}{Remark}

\begin{document}

\title
{\normalsize{\textbf{ON SUSLIN HOMOLOGY WITH INTEGRAL COEFFICIENTS   IN CHARACTERISTIC ZERO (WITH AN APPENDIX BY BRUNO KAHN)}}}

\author{ Xiaowen Hu\\
Appendix by Bruno Kahn}

\date{}
\maketitle

\begin{abstract}
We show that the Suslin homology group with integral coefficients of a scheme $X$ separated of finite type over an algebraically closed field  of characteristic 0 is a direct sum of a uniquely divisible group, finite copies of  $\mathbb{Q}/\mathbb{Z}$, and a finitely generated group. We also study the possible type of homomorphisms between such groups induced by the morphisms of schemes. An appendix written by Bruno Kahn is included, which  simplifies the proofs and generalizes the results.
\end{abstract}

\section{Introduction}

In this paper we show that for a scheme $X$ separated of finite type over an algebraically closed field $\Bbbk$ of characteristic 0, the \emph{Suslin homology} of $X$ has the following structure
\begin{equation}\label{eq-decomp-0}
H_i^S(X)=V_i\oplus D_i\oplus Z_i\oplus T_i
\end{equation}
where $V_i$ is a uniquely divisible group, $D_i$ is a divisible torsion group of the form $(\mathbb{Q}/\mathbb{Z})^s$ where $s$ is a nonnegative integer, $Z_i$ is a finite generated free abelian group, and $T_i$ is a finite abelian group. Moreover we study the possible types of homomorphisms between such groups arisen from morphisms of schemes. For the precise statements are theorems \ref{thm-mild-2} and \ref{thm-mild-3}. 


For $X$ smooth projective over $\Bbbk$, where $\Bbbk$ is an algebraically closed field of arbitrary characteristic, a decomposition of the form (\ref{eq-decomp-0}) with a slight modification on the $p$-primary part, is shown in \cite{Gei17}, and in this case $Z_i=0$. In general for singular $X$, for example certain singular curves, $Z_i$ is nonzero. It seems not easy to prove (\ref{eq-decomp-0}) by dévissage from smooth schemes (see remark \ref{rmk-mild-1}). Our approach involves some transcendental constructions, which do not work in positive characteristic (but see conjecture \ref{conj-pmild}).

To facilitate our presentation, we introduce the notion of \emph{mild abelian groups}, which means the groups having a decomposition of the form (\ref{eq-decomp-0}), and  certain lifting conditions for the homomorphisms betwee{}n these groups (definition \ref{def-mild}, \ref{def-cond-M1} and \ref{def-cond-M2}). Then we prove several criteria for the mildness of abelian groups, and the conditions (M1) and (M2) for homomorphisms.  

The proof of the main theorem consists of two steps. In the first step, for an algebraically closed subfield of $\mathbb{C}$, we use a construction of Suslin and Voevodsky \cite{SV96} to obtain a morphism from the complex of algebraic singular chains to a simplicial abelian group related to the infinite symmetric product $\coprod_{i\geq 0}\mathrm{S}^i(X(\mathbb{C}))$, then we apply one of our mildness criteria \ref{prop-comparison1}. Meanwhile we give a proof of a variant of a theorem of Dold and Thom \cite{DT58}, which was stated in \cite[theorem 8.2]{SV96} without proof.

In the second step  we use the Lefschetz principle and  the criterion \ref{prop-limit} on the mildness of the colimit of a system of mild abelian groups to obtain the main theorem. For this we need a rigidity theorem of Jannsen \cite{Jan15} to show the rigidity of Suslin homology. The original form of Jannsen's theorem does not directly apply. But a careful reading of his proof yields a modification of his notion of \emph{rigid functor}. Then we construct an ad hoc bivariant homology theory so that we  can apply Jannsen's theorem. 

After 
this work 
was done, Bruno Kahn told me a conceptual way to define  mild abelian groups, 
to understand 
Conditions (M1) and (M2), 
that the proof of the main theorem can be drastically simplified by using certain results of 
Voevodsky, 
Ayoub, Cisinski and D\'{e}glise,  and
that 
Conjecture \ref{conj-cor-hom} can also be confirmed. His 
arguments are 
included as 
 an appendix.

\section{Mild abelian groups}
In this section, after recalling some standard facts on abelian groups, we introduce the notion of mild abelian groups, and prove several  properties that we will need later.  All abelian groups are additively written, and thus for $n\in \mathbb{Z}$ and $g$ an element of an abelian group $G$, $ng$ means the sum of $n$ copies of $g$. We abbreviate the subscript or superscript $\mathbb{Z}$ in $\Hom_{\mathbb{Z}}(\cdot,\cdot)$, $\Ext^1_{\mathbb{Z}}(\cdot,\cdot)$ and $\Tor^{\mathbb{Z}}_1(\cdot,\cdot)$.
\subsection{Definitions}
Our reference on infinite abelian groups is \cite[Chap. 4]{Rob96}.
\begin{definition}\label{def-mild}
Let $G$ be an abelian group.
\begin{enumerate}
\item An element $g$ of $G$ is called \emph{divisible} if for any positive integer $n$ there exists $h\in G$ such that $g=nh$. The subgroup of all divisible elements of $G$ is denoted by $G_\mathrm{div}$. 
\item An element $g$ of $G$ is called \emph{uniquely divisible} if for any positive integer $n$ there exists a unique $h\in G$ such that $g=nh$. 
\end{enumerate}
\end{definition}

We refer the reader to \cite[4.1.4, 4.1.5]{Rob96} for the following proposition.
\begin{proposition}\label{abelian1}
Let $G$ be an abelian group. Then $G_\mathrm{div}$ is a direct summand of $G$. For every prime number $p$, let $G_{(p^\infty)}$ be the subgroup consisting of $p^m$-torsion elements of $G_\mathrm{div}$, with $m$  running over all positive integers. Then the uniquely divisible elements of $G$ form a subgroup $G_{\mathrm{ud}}$, which is isomorphic to a  $\mathbb{Q}$-vector space. There is a canonical decomposition
\begin{eqnarray*}
G_\mathrm{div}=G_{\mathrm{ud}} \oplus \bigoplus_{p}G_{(p^\infty)},
\end{eqnarray*}
where $p$ is running over all prime numbers.\pqed
\end{proposition}

For an abelian group $G$, we denote the torsion subgroup of $G$ by $G_\mathrm{tor}$,  and the subgroup of all divisible torsion elements by $G_{\mathrm{td}}$, and the subgroup of all uniquely divisible elements by $G_{\mathrm{ud}}$ (this is a group by proposition \ref{abelian1} (2)). We have the following diagram
\begin{eqnarray}
\begin{gathered}\label{diag-1}
\xymatrix{
G_\mathrm{tor}\cap G_\mathrm{div}=G_\mathrm{td}\ar@{^{(}->}[d] \ar@{^{(}->}[r] & G_\mathrm{tor} \ar@{^{(}->}[d]\\
 G_\mathrm{div} \ar@{^{(}->}[r] & G
}
\end{gathered}
\end{eqnarray}

\begin{definition}\label{abelian2}
Let $G$ be an abelian group. We say that $G$ is called \emph{mild}, if $G$ has a direct sum decomposition as follows:
\begin{eqnarray*}
G\cong V(G) \oplus D(G) \oplus Z(G) \oplus T(G),
\end{eqnarray*}
where 
\begin{enumerate}
  \item[(i)] $V$ is uniquely divisible;
  \item[(ii)] $D(G)$ is isomorphic to a divisible torsion group of the form $(\mathbb{Q}/\mathbb{Z})^{\oplus s}$, where $s$ is a non-negative integer;
  \item[(iii)] $Z(G)$ is a finite generated free abelian group;
  \item[(iv)] $T(G)$ is a finite abelian group.
\end{enumerate}
For a mild abelian group $G$, we denote the number $s$ in the condition (ii) by $s(G)$, and the rank of $Z(G)$ by $r(G)$.
\end{definition}

The  decomposition in the above definition, when exists, is of course not canonical. But there are canonical filtrations as indicated in the diagram (\ref{diag-1}).\\

Since we have an exact sequence
\[
0\rightarrow \Hom(\mathbb{Q},\mathbb{Q})\rightarrow \Hom(\mathbb{Q},\mathbb{Q}/\mathbb{Z})\rightarrow \Ext^1(\mathbb{Q},\mathbb{Z})\rightarrow 0, 
\]
in general a homomorphism $\mathbb{Q}\rightarrow \mathbb{Q}/\mathbb{Z}$ cannot lift  to a homomorphism $\mathbb{Q}\rightarrow \mathbb{Q}$. Similarly, for two primes $p,q$, 
\[
\Hom(\mathbb{Q}_p/\mathbb{Z}_p, \mathbb{Q}_q/\mathbb{Z}_q)=\delta_{p,q}\mathbb{Z}_p,
\]
thus a homomorphism $\mathbb{Q}/\mathbb{Z}\rightarrow\mathbb{Q}/\mathbb{Z}$ in general does not lift to a homomorphism $\mathbb{Q}\rightarrow \mathbb{Q}$. From certain geometric reasons, we introduce some \emph{liftable conditions} for the homomorphisms between mild abelian groups.

\begin{definition}\label{def-cond-M1}
Let $G_1, G_2$ be mild abelian groups, $f:G_1\rightarrow G_2$ a homomorphism. We say that $f$ \emph{satisfies the condition (M1)}, if there is a $\mathbb{Q}$-vector space $W_i$ of rank $s(G_i)$, a full rank $\mathbb{Z}$-lattice $\Gamma_i$ of $W_i$, and an isomorphism $h_i:W_i/\Gamma_i\xrightarrow{\sim} (G_i)_{\mathrm{td}}$, for $i=1,2$, such that there exists a commutative diagram
\begin{equation}
\xymatrix{
  V(G_1)\oplus W_1 \ar[d]_{(\id_{V(G_1)},p_1)} \ar[r]^{\tilde{f}} & V(G_2)\oplus W_2 \ar[d]^{(\id_{V(G_2)},p_2)} \\
  (G_1)_{\mathrm{div}} \ar[r]^{f} & (G_2)_{\mathrm{div}},
}
\end{equation} 
where $p_i$ is the map $W_i\twoheadrightarrow W_i/\Gamma_i\xrightarrow{h_i}(G_i)_{\mathrm{td}}$, for $i=1,2$, and $\tilde{f}$ is a homomorphism of abelian groups (equivalently, a homomorphism of $\mathbb{Q}$-vector spaces).
\end{definition}

\begin{definition}\label{def-cond-M2}
Let $G$ be a mild abelian group, $f:G\rightarrow G$ an endomorphism. We say that $f$ \emph{satisfies the condition (M2)}, if there is a $\mathbb{Q}$-vector space $W$ of rank $s(G)$, a full rank $\mathbb{Z}$-lattice $\Gamma$ of $W$, and an isomorphism $h:W/\Gamma\xrightarrow{\sim} (G)_{\mathrm{td}}$, such that there exists a commutative diagram
\begin{equation}
\xymatrix{
  V(G)\oplus W \ar[d]_{(\id_{V(G)},p)} \ar[r]^{\tilde{f}} & V(G)\oplus W \ar[d]^{(\id_{V(G)},p)} \\
  (G)_{\mathrm{div}} \ar[r]^{f} & (G)_{\mathrm{div}},
}
\end{equation} 
where $p$ is the map $W\twoheadrightarrow W/\Gamma\xrightarrow{h}(G)_{\mathrm{td}}$, and $\tilde{f}$ is an endomorphism of abelian groups (equivalently, a homomorphism of $\mathbb{Q}$-vector spaces).
\end{definition}

Let $G$ be a mild abelian group with $V(G)=Z(G)=T(G)=0$, thus $\mathrm{End}(G)\cong \prod_{p\in \mathbf{P}}\mathrm{GL}_{s(G)}(\mathbb{Z}_p)$, where $\mathbf{P}$ is the set of prime numbers. Then an endomorphism $f$ of $G$ can be identified to $(f_p)_{p\in \mathbf{P}}$ where $f_p\in \mathrm{GL}_{s(G)}(\mathbb{Z}_p)$. In this case the following lemma relates the  condition (M2) to a certain $\ell$-independence property.

\begin{lemma}
The endomorphism $f$ satisfies  (M2) if and only if there exists a matrix $M\in \mathrm{GL}_{s(G)}(\mathbb{Z})$, such that $f_p$ is similar to $M$ in $\mathrm{GL}_{s(G)}(\mathbb{Z}_p)$, for every prime $p$.\pqed
\end{lemma}

\subsection{Some criteria for mildness}
The following lemma follows immediately from the definition.
\begin{lemma}\phantomsection\label{lem-mild-1}
\begin{enumerate}
  \item[(i)] Let $A$ be a mild abelian group, $B$  a divisible torsion mild abelian group, $f:A\rightarrow B$  an injective homomorphism. Then $\mathrm{coker}(f)$ is mild.
  \item[(ii)] Let $A$ be a mild abelian group, $B$  a uniquely divisible  abelian group, $f:A\rightarrow B$  an injective homomorphism. Then $\mathrm{coker}(f)$ is mild.
  \item[(iii)] $G$ is mild if and only if, $G/G_{\mathrm{div}}$ is finitely generated and there exists an integer $s$ such that the number of the $n$-torsion elements
of  $G_{\mathrm{div}}$ is equal to $ns$ for any positive integer $n$.
\end{enumerate}\pqed
\end{lemma}

\begin{proposition}\label{prop-mild-1}
\begin{enumerate}
\item[(i)] If $A$ is mild and $B$ is mild with $s(B)=0$, and $f:A\rightarrow B$ a homomorphism, then $\ker(f)$ is mild.
\item[(ii)] If there is a short exact sequence of abelian groups 
\begin{eqnarray*}
0\rightarrow A\rightarrow B\rightarrow C\rightarrow 0,
\end{eqnarray*} 
and  $A,B$ are mild, then $C$ is mild, and
\begin{gather}\label{eq-prop-mild-1}
r(C)\leq r(B)\leq r(A)+r(C),\notag\\
s(A)\leq s(B)\leq s(A)+s(C),\\
r(A)+r(C)-r(B)=s(A)+s(C)-s(B).\notag
\end{gather}

\end{enumerate}
\end{proposition}
Proof: (i) Since $s(B)=0$, $A_{\mathrm{td}}$ maps to zero, and since it is an injective abelian group, we have $\mathrm{ker}(f)=A_{\mathrm{td}}\oplus \mathrm{ker}(f_1)$, where $f_1:A/A_{\mathrm{td}}\rightarrow B$. Thus we can assume that $A_{\mathrm{td}}=0$. Then $A_{\mathrm{div}}$ is uniquely divisible, and we have a commutative diagram with exact rows
\[
\xymatrix{
  0 \ar[r] & A_{\mathrm{div}} \ar[r] \ar[d] & A \ar[r] \ar[d] & A/A_{\mathrm{div}} \ar[r] \ar[d] & 0 \\
  0 \ar[r] & B_{\mathrm{div}} \ar[r]  & B \ar[r]  & B/B_{\mathrm{div}} \ar[r]  & 0 
}
\]
which induces an exact sequence
\[
0\rightarrow C_0\rightarrow \mathrm{ker}(f)\rightarrow C_1
\]
where $C_0=\mathrm{ker}(A_{\mathrm{div}}\rightarrow B_{\mathrm{div}})$, and $C_1=\mathrm{ker}(A/A_{\mathrm{div}}\rightarrow B/B_{\mathrm{div}})$. Since $C_1$ is a finitely generated abelian group,   $C_2=\mathrm{Im}(\mathrm{ker}(f)\rightarrow C_1)$ is also finitely generated; $C_0$ is  the direct sum of $A_{\mathrm{td}}$ and the kernel of a map of $\mathbb{Q}$-vector spaces which is thus uniquely divisible, so $C_0$ is mild and is an injective abelian group.  Hence $\mathrm{ker}(f)=C_0\oplus C_2$ is mild.

(ii) A homomorphism $f:A\rightarrow B$ maps $A_{\mathrm{td}}$ into $B_{\mathrm{td}}$, so we have a commutative diagram with exact rows 
\[
\xymatrix{
  0 \ar[r] & A_{\mathrm{td}} \ar[r] \ar[d] & A \ar[r] \ar[d] & A/A_{\mathrm{td}} \ar[r] \ar[d] & 0 \\
  0 \ar[r] & B_{\mathrm{td}} \ar[r]  & B \ar[r]  & B/B_{\mathrm{td}} \ar[r]  & 0 
}
\]
which induces an exact sequence
\[
0\rightarrow C_0\rightarrow C_1\rightarrow C\rightarrow C_2\rightarrow 0,
\]
where $C_0=\mathrm{ker}(A/A_{\mathrm{td}}\rightarrow B/B_{\mathrm{td}})$, $C_1=\mathrm{coker}(A_{\mathrm{td}}\rightarrow B_{\mathrm{td}})$, and $C_2=\mathrm{coker}(A/A_{\mathrm{td}}\rightarrow B/B_{\mathrm{td}})$. Then $C_1/C_0$ is a divisible torsion group,  so $C\cong (C_1/C_0)\oplus C_2$. By (i) $C_0$ is mild, so by lemma \ref{lem-mild-1} $C_1/C_0$ is mild. It suffices to show that $C_2$ is mild. So in the beginning we can assume that $A_{\mathrm{td}}=B_{\mathrm{td}}=0$. Then $A_{\mathrm{div}}$ and $B_{\mathrm{div}}$ are uniquely divisible, and we have 
a commutative diagram with exact rows 
\[
\xymatrix{
  0 \ar[r] & A_{\mathrm{div}} \ar[r] \ar[d] & A \ar[r] \ar[d] & A/A_{\mathrm{div}} \ar[r] \ar[d] & 0 \\
  0 \ar[r] & B_{\mathrm{div}} \ar[r]  & B \ar[r]  & B/B_{\mathrm{div}} \ar[r]  & 0 
}
\]
which induces an exact sequence
\[
0\rightarrow C_0\rightarrow C_1\rightarrow C\rightarrow C_2\rightarrow 0,
\]
where $C_0=\mathrm{ker}(A/A_{\mathrm{div}}\rightarrow B/B_{\mathrm{div}})$, $C_1=\mathrm{coker}(A_{\mathrm{div}}\rightarrow B_{\mathrm{div}})$, and $C_2=\mathrm{coker}(A/A_{\mathrm{div}}\rightarrow B/B_{\mathrm{div}})$. Then $C_1/C_0$ is divisible, and thus $C\cong (C_1/C_0)\oplus C_2$. But $C_0$ and $C_2$ are finitely generated abelian groups. By lemma \ref{lem-mild-1} (ii) $C_1/C_0$ is mild. So $C$ is mild. The inequalities and equalities (\ref{eq-prop-mild-1}) also follows.\pqed

\begin{remark}\label{rmk-mild-1}
Mildness do not preserve under extensions. For example, $\prod_{p}\mathbb{Z}_p/\mathbb{Z}$ is uniquely divisible, while $\prod_{p}\mathbb{Z}_p$ is not mild. This prohibits some dévissage attempts to show the mildness of  certain (co)homology groups.
\end{remark}

In general, the image of a homomorphism of mild abelian groups is not mild. But we have:
\begin{proposition}\label{prop-mild-2}
Let $G_1$, $G_2$ be mild abelian groups, and $f:G_1\rightarrow G_2$ a homomorphism satisfying the condition (M1). Then the image of $f$ is mild.
\end{proposition}
Proof: Since $f$ maps $(G_1)_{\mathrm{div}}$ into $(G_2)_{\mathrm{div}}$, it induces a map $\bar{f}:G_1/(G_1)_{\mathrm{div}}\rightarrow G_2/(G_2)_{\mathrm{div}}$. Moreover, $f((G_1)_{\mathrm{div}})$ is divisible, so we have a decomposition $\mathrm{Im}(f)=f((G_1)_{\mathrm{div}})\oplus \mathrm{Im}(\bar{f})$. Since $\mathrm{Im}(\bar{f})$ is a finitely generated abelian group, it reduces to consider  the case $G_2$ is divisible and mild. Now we use the notations in the definition \ref{def-cond-M1}. We denote the image of $\tilde{f}$ by $U$. The intersection $U\cap \Gamma_2$ is a lattice in $U$, not of full rank in general, and we denote it by $\Gamma$. Denote by $U_0$ the $\mathbb{Q}$-subspace of $U$ spanned by $\Gamma$. Then $\mathrm{Im}(f)\cong U/U_0\oplus U_0/\Gamma$, which is mild. \pqed

\subsubsection{Colimits of mild abelian groups}
\begin{proposition}\label{prop-limit}
Let $(G_i)_{i\in I}$ be a \emph{filtered} system of mild abelian groups, satisfying for the homomorphism $\varphi_{\nu}:G_i\rightarrow G_j$ corresponding to  any arrow $\nu$ in  $I$, and any integer $n$, the induced homomorphisms
\begin{equation}\label{eq-limit-0}
\varphi_{\nu}[n]:G_i[n]\rightarrow G_j[n],\ \varphi_{\nu}/n:G_i/n\rightarrow G_j/n
\end{equation}
are isomorphisms. Then the colimit $G=\varinjlim_{i\in I}G_i$ is mild, and for any $i\in I$
there are canonical isomorphisms
\begin{equation}\label{eq-limit-0.1}
(G_i)_{\mathrm{td}}\xrightarrow{\sim} G_{\mathrm{td}},\
(G_i)_{\mathrm{tor}}\xrightarrow{\sim} G_{\mathrm{tor}},\
Z(G_i)\xrightarrow{\sim} Z(G)
\end{equation}
and  a (non-canonical) isomorphism
\begin{equation}\label{eq-limit-0.2}
G_i/(G_i)_{\mathrm{ud}}\xrightarrow{\sim} G/G_{\mathrm{ud}}.
\end{equation}
\end{proposition}
Here for the notion of a \emph{filtered system}, we mean \cite[I, definition 2.7]{SGA4}. The conclusion is not valid for non-filtered colimits in general.

Proof: A homomorphism $\varphi:G_i\rightarrow G_j$ in this system maps $(G_i)_{\mathrm{td}}$ into $(G_i)_{\mathrm{td}}$. Since filtered colimits preserve exactness, we have a short exact sequence
\begin{equation}\label{eq-limit-1}
0\rightarrow\varinjlim_{I}\big((G_i)_{\mathrm{td}}\big)\rightarrow \varinjlim_{I}G_i\rightarrow \varinjlim_{I}\big(G_i/(G_i)_{\mathrm{td}}\big)\rightarrow 0,
\end{equation}
Since any homomorphism $\varphi$ in the system induce isomorphisms $G_i[n]\xrightarrow{\sim}G_j[n]$, all the arrows in the system $\big((G_i)_{\mathrm{td}}\big)_{i\in I}$ are isomorphisms, so we have canonical isomorphisms
\[
(G_i)_{\mathrm{td}}\xrightarrow{\sim }\varinjlim_{I}(G_i)_{\mathrm{td}}.
\]
Therefore $\varinjlim_{I}(G_i)_{\mathrm{td}}$ is divisible, and is thus an injective abelian group. Hence the sequence (\ref{eq-limit-1}) splits. 

Now for any $i\in I$, the divisible subgroup of $G_i/(G_i)_{\mathrm{td}}$ is isomorphic to $G_{\mathrm{ud}}$, and the homomorphism $G_i/(G_i)_{\mathrm{td}}\rightarrow G_j/(G_j)_{\mathrm{td}}$ induced by $\varphi$ maps the divisible subgroup of the former into the divisible subgroup of the latter. Since filtered colimits preserve exactness, we have a short exact sequence
\begin{equation}\label{eq-limit-2}
0\rightarrow\varinjlim_{I}\big((G_i)_{\mathrm{ud}}\big)\rightarrow \varinjlim_{I}\big(G_i/(G_i)_{\mathrm{td}}\big)\rightarrow \varinjlim_{I}\big(G_i/(G_i)_{\mathrm{div}}\big)\rightarrow 0.
\end{equation}
The group $\varinjlim_{I}\big((G_i)_{\mathrm{ud}}\big)$ is still uniquely divisible, and is injective. So (\ref{eq-limit-2}) splits. 

Finally, the torsion subgroup of $G_i/(G_i)_{\mathrm{div}}$ is isomorphic to $T_i=(G_i)_{\mathrm{tor}}/(G_i)_{\mathrm{td}}$, the finite torsion part of $G_i$, and the homomorphism $G_i/(G_i)_{\mathrm{div}}\rightarrow G_j/(G_j)_{\mathrm{div}}$ induced by $\varphi$ maps the torsion subgroup of the former into the torsion subgroup of the latter.
Since filtered colimits preserve exactness, we have a short exact sequence
\begin{equation}\label{eq-limit-3}
0\rightarrow\varinjlim_{I}T_i\rightarrow \varinjlim_{I}\big(G_i/(G_i)_{\mathrm{div}}\big)\rightarrow \varinjlim_{I}Z_i\rightarrow 0,
\end{equation} 
where $Z_i=G_i/((G_i)_{\mathrm{div}}+(G_i)_{\mathrm{tor}})$. By the condition (\ref{eq-limit-0}), each  homomorphism in the system $(T_i)_{i\in I}$ and the system $(Z_i)_{i\in I}$ is an isomorphism. So $T_i\cong \varinjlim_{I}T_i$ and $Z_i\cong \varinjlim_{I}Z_i$. In particular $\varinjlim_{I}Z_i$ is free of finite rank, and (\ref{eq-limit-3}) splits. So we obtain a (non-canonical) decomposition
\[
\varinjlim_{I}G_i\cong V\oplus Z_i\oplus D_i\oplus T_i,
\]
where $D_i\cong (G_i)_{\mathrm{td}}$ is a canonical isomorphism, and $V$ is a uniquely divisible group.\pqed

\begin{proposition}\label{prop-limit-hom}
Let $(G_i)_{i\in I}$ and $(H_i)_{i\in I}$ be two filtered systems of mild abelian groups, and $(f_i:G_i\rightarrow H_i)_{i\in I}$ a system of homomorphisms, satisfying
\begin{enumerate}
   \item[(a)] for the homomorphism $\varphi_{\nu}:G_i\rightarrow G_j$ corresponding to  any arrow $\nu$ in $I$, and any integer $n$, the induced homomorphisms
\begin{equation}\label{eq-prop-limit-hom-1}
\varphi_{\nu}[n]:G_i[n]\rightarrow G_j[n],\ \varphi_{\nu}/n:G_i/n\rightarrow G_j/n
\end{equation}
are isomorphisms, and $\varphi_{\nu}$ maps torsion free elements to torsion free elements;
  \item[(b)] for the homomorphism $\psi_{\nu}:H_i\rightarrow H_j$ corresponding to  any arrow $\nu$ in $I$, and any integer $n$, the induced homomorphisms
\begin{equation}\label{eq-limit-hom-2}
\psi_{\nu}[n]:H_i[n]\rightarrow H_j[n],\ \psi_{\nu}/n: H_i/n\rightarrow H_j/n
\end{equation}
are isomorphisms, and $\psi_{\nu}$ maps torsion free elements to torsion free elements;
  \item[(c)] for any object $i$ of $I$, the homomorphism $f_i:G_i\rightarrow H_i$ satisfies (M1).
 \end{enumerate} 
 Then the colimits $G=\varinjlim_{i\in I}G_i$ and $H=\varinjlim_{i\in I}H_i$ are mild, and the induced homomorphism $f:G\rightarrow H$ satisfies (M1).
\end{proposition}
\emph{Proof:} By \ref{prop-limit}, $G$ and $H$ are mild abelian groups, and there are canonical isomorphisms $D(G_i)\xrightarrow{\sim} D(G)$ and $D(H_i)\xrightarrow{\sim} D(H)$. Thus we can identify $D(G_i)$ to $D(G)$, and $D(H_i)$ to $D(H)$, for every $i\in I$.
 Suppose $s=s(D(G))$, $t=s(D(H))$. Since for all arrows $\nu$ in $I$, $\varphi_{\nu}$ and $\psi_{\nu}$ map torsion free elements to torsion free elements, there are canonical decompositions
 \[
 G\cong D(G)\oplus \varinjlim_{i\in I}V(G_i),\ H\cong D(H)\oplus \varinjlim_{i\in I}V(H_i).
 \]
So there is a $\mathbb{Q}$-vector space $W_i$ of finite dimension $s$, a full rank lattice $\Gamma_i$ in $W_i$, and an isomorphism $g_i:W_i/\Gamma_i\cong D(G)$, and a $\mathbb{Q}$-vector space $U_i$ of finite dimension $t$, a full rank lattice $\Lambda_i$ in $W_i$, and an isomorphism $h_i:U_i/\Lambda_i\cong D(H)$, and moreover a commutative diagram
\[
\xymatrix{
  V(G_i)\oplus W_i \ar[d]_{(\id_{V(G_i)},p_i)} \ar[r]^{\tilde{f}_i} & V(H_i)\oplus U_i \ar[d]^{(\id_{V(H_i)},q_i)} \\
  (G_i)_{\mathrm{div}} \ar[r]^{f_i} & (H_i)_{\mathrm{div}},
}
\]
where $p_i$ is the composition $W_i\twoheadrightarrow W_i/\Gamma_i\xrightarrow{g_i} D(G)$, and $q_i$ is the composition $U_i\twoheadrightarrow U_i/\Lambda_i\xrightarrow{h_i} D(H)$. By proposition \ref{prop-mild-2}, the image of the composition 
\begin{equation}\label{eq-prop-limit-hom-2}
V(G_i)\hookrightarrow (G_i)_{\mathrm{div}}\xrightarrow{f_i}(H_i)_{\mathrm{div}}\twoheadrightarrow D(H_i)=D(H)
\end{equation}
is a divisible torsion mild abelian subgroup of $D(H)$. But an ascending chain of divisible torsion mild abelian subgroups of $D(H)$ stabilizes, so the image  of the induced map $D(G)\rightarrow D(H)$ is a divisible torsion mild abelian subgroup of $D(H)$. We denote this image by $D$, and denote  $s(D)$ by $t_{\infty}$. Since the system $I$ is filtered, to show that $f$ satisfies (M1), we can consider only the cofinal sub-system $J$ of $I$ such that for $i\in J$ the image of the composition (\ref{eq-prop-limit-hom-2}) is equal to $D$. Note that $J$ is still filtered. For $i\in J$, we set $Y_i=q_i^{-1}(D)$, thus $\phi_i:=\tilde{f}_i|_{V(G_i)}:V(G_i)\rightarrow Y_i$ is surjective. For any $i,j\in J$, and an arrow $\nu:i\rightarrow j$ in $J$, there is a commutative diagram
\begin{equation}
\begin{gathered}\label{eq-prop-limit-hom-3}
\xymatrix{
V(G_i) \ar[d] \ar[r]^{\phi_i} & Y_i \ar[r] & D \\
V(G_j) \ar[r]^{\phi_j} & Y_j \ar[ru] &.
}
\end{gathered}
\end{equation}
Taking a basis $e_1,...,e_{t_{\infty}}$ of $Y_i$, and choosing their arbitrary preimages under $\phi_i$, one can find a map $\eta_{\nu}:Y_i\rightarrow Y_j$ completing the above diagram. Since $Y_i$ and $Y_j$ are both of rank $t_{\infty}$, and map surjectively onto $D$, $\eta_{\nu}$ is surjective and is thus an isomorphism of $\mathbb{Q}$-vector spaces. Moreover, there is an exact sequence
\[
0\rightarrow \mathrm{Hom}(Y_i,\Lambda_i\cap Y_j)\rightarrow \mathrm{Hom}(Y_i,Y_j)\rightarrow \mathrm{Hom}(Y_i,D).
\]
Since $\mathrm{Hom}(Y_i,\Lambda_i\cap Y_j)=0$, the homomorphism $\eta_{\nu}$ is unique. Therefore there is a canonical isomorphism $\eta_{\nu}:Y_i\xrightarrow{\sim}Y_j$ induced by the arrow $\nu$ in $J$. Since $J$ is filtered, there is a canonical isomorphism $Y_i\cong Y_j$ for any $i,j\in J$. Now fix $i_0\in J$. Then for any $i\in J$, the map (\ref{eq-prop-limit-hom-2}) factors through $Y_{i_0}\twoheadrightarrow D$, compatibly with the transition maps $V(G_i)\rightarrow V(G_j)$. So there is a homomorphism $\varinjlim_{i\in J}V(G_i)\rightarrow Y_{i_0}$ making the following diagram commutes:
\[
\xymatrix{
  \varinjlim_{i\in J}V(G_i) \ar[r] & Y_{i_0} \ar[d] \ar@{^{(}->}[r] & U_{i_0} \ar[d]  \\
  V(G_i) \ar[u] \ar[r] & D \ar@{^{(}->}[r] & D(H).
}
\]
We take $W=W_{i_0}$, $U=U_{i_0}$, there  thus exists a commutative diagram
\[
\xymatrix{
  \varinjlim_{i\in J}V(G_i) \oplus W \ar[r] \ar[d] & \varinjlim_{i\in J}V(H_i)\oplus U \ar[d] \\
  G_{\mathrm{div}} \ar[r] & H_{\mathrm{div}}
}
\]
satisfying  definition \ref{def-cond-M1}.
\pqed
\begin{proposition}\label{prop-limit-endo}
Let $(G_i)_{i\in I}$ be a filtered systems of mild abelian groups, and $(f_i:G_i\rightarrow G_i)_{i\in I}$ a system of endomorphisms, satisfying
\begin{enumerate}
   \item[(a)] for the homomorphism $\varphi_{\nu}:G_i\rightarrow G_j$ corresponding to  any arrow $\nu$ in $I$, and any integer $n$, the induced homomorphisms
\begin{equation}\label{eq-limit-endo-0}
\varphi_{\nu}[n]:G_i[n]\rightarrow G_j[n],\ \varphi_{\nu}/n:G_i/n\rightarrow G_j/n
\end{equation}
are isomorphisms, and $\varphi_{\nu}$ maps torsion free elements to torsion free elements;
  \item[(b)] for any object $i$ of $I$, the endomorphism $f_i:G_i\rightarrow G_i$ satisfies (M2).
 \end{enumerate} 
 Then the colimit $G=\varinjlim_{i\in I}G_i$ is mild, and the induced endomorphism $f:G\rightarrow G$ satisfies (M2).
\end{proposition}
Proof: The proof is similar to that of proposition \ref{eq-prop-limit-hom-2}. We omit it.\pqed

\subsubsection{Mildness via quasi-isomorphisms after tensoring \texorpdfstring{$\mathbb{Z}/n$}{}}
Let $K_\bullet$ be a chain complex of free abelian groups, i.e., we have a family of free abelian groups $\{K_i\}_{i\in \mathbb{Z}}$  and group homomorphisms $d:K_i\rightarrow K_{i-1}$ such that $d^2=0$. For every abelian group $A$, by the universal coefficient theorem, there are (non-canonically) split exact sequences
\begin{eqnarray}\label{uct1}
0\rightarrow H_i(K_\bullet)\otimes A\rightarrow H_i(K_\bullet\otimes A)\rightarrow\Tor_{1}(H_{i-1}(K_\bullet), A)\rightarrow 0
\end{eqnarray}
and
\begin{eqnarray}\label{uct2}
0\rightarrow \Ext^{1}(H_{i-1}(K_\bullet), A)\rightarrow H^{i}(\Hom(K_\bullet,A))\rightarrow \Hom(H_i(K_\bullet),A)\rightarrow 0.
\end{eqnarray}

\begin{proposition}\label{prop-comparison1}
Let $K_\bullet$ and $L_\bullet$ be two chain complexes of free abelian groups, and $\phi: K_\bullet\rightarrow L_\bullet$ a homomorphism of chain complexes. Suppose that $\phi\otimes \mathbb{Z}/n$ are quasi-isomorphisms for every $n\in\mathbb{Z}$, and $H_i(L_\bullet)$  are mild abelian groups for all $i$. Then for any $i\in\mathbb{Z}$, $H_{i}(K_\bullet)$ is a mild abelian group, and 
\begin{eqnarray}\label{eq-comparison1-1}
s(H_{i-1}(K_\bullet))+r(H_{i}(K_\bullet))=s(H_{i-1}(L_\bullet))+r(H_{i}(L_\bullet)).
\end{eqnarray}
 Moreover, we have canonical isomorphisms
\begin{eqnarray}\label{eq-comparison1-2}
T(H_i(K_\bullet))\cong T(H_i(L_\bullet)).
\end{eqnarray}
\end{proposition}
Proof:  Let $C(\varphi)$ be a cone of $\varphi$. Then $C(\varphi)$ is also a complex of free abelian groups, and $C(\varphi)\otimes \mathbb{Z}/n$ is a cone of $\varphi\otimes \mathbb{Z}/n$ for every $n\in \mathbb{Z}$. One has a long exact sequence
\begin{multline*}
\cdots \rightarrow H_{i}(K_\bullet\otimes \mathbb{Z}/n) \rightarrow H_{i}(L_\bullet\otimes \mathbb{Z}/n)
\rightarrow H_{i}(C(\varphi)\otimes \mathbb{Z}/n)\rightarrow\\
 H_{i-1}(K_\bullet\otimes \mathbb{Z}/n)\rightarrow
H_{i-1}(L_\bullet\otimes \mathbb{Z}/n)\rightarrow\cdots
\end{multline*}
So $H_{i}(C(\varphi)\otimes \mathbb{Z}/n)=0$, which by  universal coefficient theorem (\ref{uct1}) implies that $H_{i}(C(\varphi))$ is uniquely divisible. Thus the image of $H_{i+1}(C(\varphi))$ in $H_{i}(K_\bullet)$ is divisible, and  
\[
H_{i}(K_\bullet)/\mathrm{Im}(H_{i+1}(C(\varphi)))
\]
maps injectively into $H_{i}(L_\bullet)$. Since $H_i(C(\phi))$ is uniquely divisible and $H_i(L_\bullet)$ is mild, by proposition \ref{prop-mild-1} (i) $H_{i}(K_\bullet)/\mathrm{Im}(H_{i+1}(C(\varphi)))$ is mild. Since $\mathrm{Im}(H_{i+1}(C(\varphi)))$ is an injective abelian group, one has
\[
H_i(K_\bullet)=\mathrm{Im}(H_{i+1}(C(\varphi)))\oplus H_{i}(K_\bullet)/\mathrm{Im}(H_{i+1}(C(\varphi))).
\]
By (\ref{uct1}) and lemma \ref{lem-mild-1} (iii) one easily sees that  $H_{i}(K_\bullet)$ is mild.
\pqed

\begin{proposition}\label{prop-comparison2}
Let $K_\bullet$, $L_\bullet$, $K'_\bullet$, $L'_\bullet$ be  chain complexes of free abelian groups, and
\begin{equation}
\begin{gathered}\label{eq-comparison2-1}
\xymatrix{
  K_\bullet \ar[r]^{\phi} \ar[d]_{\kappa} & L_\bullet \ar[d]^{\lambda} \\
  K'_\bullet \ar[r]^{\phi'} & L'_\bullet
}
\end{gathered}
\end{equation}
be a commutative diagram of morphisms of complexes.  Suppose that $\phi\otimes \mathbb{Z}/n$ and $\phi'\otimes \mathbb{Z}/n$ are quasi-isomorphisms for every $n\in\mathbb{Z}$, and $H_i(L_\bullet)$ and $H_i(L'_\bullet )$ are finitely generated abelian groups for all $i$. Then for any $i\in\mathbb{Z}$, $H_{i}(K_\bullet)$ and $H_i(K'_\bullet)$ are mild abelian groups, and the induced homomorphism $H_i(\kappa):H_{i}(K_\bullet)\rightarrow H_i(K'_\bullet)$ satisfies the condition (M1).  
\end{proposition}
Proof: The first statement is shown in proposition \ref{prop-comparison1}. For the second statement, consider  cones of $\phi$ and $\phi'$. There is a commutative diagram
\begin{equation}
\begin{gathered}\label{eq-comparision2-2}
\xymatrix{
  H_{i+1}(L_\bullet) \ar[d]_{H_{i+1}(\lambda)} \ar[r]^{\theta_{i+1}} &
 H_{i+1}(C(\varphi)) \ar[d] \ar[r]^{\tau_{i+1}} &
 H_{i}(K_\bullet) \ar[r] \ar[d]_{H_{i}(\kappa)} &
H_{i}(L_\bullet) \ar[d]_{H_i(\lambda)} \\
H_{i+1}(L'_\bullet)  \ar[r]^{\theta'_{i+1}} &
 H_{i+1}(C(\varphi'))  \ar[r]^{\tau'_{i+1}} &
 H_{i}(K'_\bullet) \ar[r]  &
H_{i}(L'_\bullet) 
}
\end{gathered}
\end{equation}
and by the proof of proposition \ref{prop-comparison1}, $H_{i+1}(C(\varphi))$ and $H_{i+1}(C(\varphi'))$ are uniquely divisible. By our assumption $H_{i}(L_\bullet)$ and $H_{i}(L'_\bullet)$ are finitely generated, thus $\tau_{i+1}$ maps $H_{i+1}(C(\varphi))$ onto $H_{i}(K_\bullet)_{\mathrm{div}}$, (resp., $\tau'_{i+1}$ maps $H_{i+1}(C(\varphi'))$ onto $H_{i}(K'_\bullet)_{\mathrm{div}}$), and the image of $\theta_{i+1}$ (resp. the image of $\theta'_{i+1}$) is a lattice in $H_{i+1}(C(\varphi))$ (a lattice in $H_{i+1}(C(\varphi'))$). Thus by definition $H_i(\kappa)$ satisfies the condition (M1).\pqed

\begin{proposition}\label{prop-comparison3}
Let $K_\bullet$, $L_\bullet$ be  chain complexes of free abelian groups, and
\begin{equation}
\begin{gathered}\label{eq-comparison3-1}
\xymatrix{
  K_\bullet \ar[r]^{\phi} \ar[d]_{\kappa} & L_\bullet \ar[d]^{\lambda} \\
  K_\bullet \ar[r]^{\phi} & L_\bullet
}
\end{gathered}
\end{equation}
be a commutative diagram of morphisms of complexes.  Suppose that $\phi\otimes \mathbb{Z}/n$ are quasi-isomorphisms for every $n\in\mathbb{Z}$, and $H_i(L_\bullet)$  are finitely generated abelian groups for all $i$. Then for any $i\in\mathbb{Z}$, $H_{i}(K_\bullet)$ is a mild abelian groups, and the induced endomorphism $H_i(\kappa):H_{i}(K_\bullet)\rightarrow H_i(K_\bullet)$ satisfies the condition (M2).  
\end{proposition}
Proof: The proof is similar to the proof of \ref{prop-comparison2}, and we omit it. \pqed

\begin{remark}
In proposition \ref{prop-comparison1}, if one reverses the direction of the homomorphism $\phi$, the conclusion on the mildness of $H_{i}(K_\bullet)$ is not true. For example, the diagonal homomorphism $\Delta: \mathbb{Z}\rightarrow \widehat{\mathbb{Z}}= \prod_{p}\mathbb{Z}_p$ becomes an isomorphism after tensoring $\mathbb{Z}/m$ for any integer $m$. Let 
\begin{eqnarray*}
0\rightarrow N\rightarrow M\rightarrow \prod_{p}\mathbb{Z}_p \rightarrow 0
\end{eqnarray*} 
be a short exact sequence of abelian groups with $M$ and $N$ free, and $f: \mathbb{Z}\rightarrow M$ be a lifting of $\Delta$. Then the morphism of complexes of free abelian groups
\begin{eqnarray*}
\xymatrix{
	\cdots \ar[r] & 0 \ar[r] \ar[d] & \cdots \ar[r] & 0\ar[d] \ar[r] & \mathbb{Z} \ar[r] \ar_{f}[d] & 0 \ar[r] \ar[d] & \cdots\\
	\cdots \ar[r] & 0 \ar[r] & \cdots \ar[r] & N \ar[r] & M \ar[r] & 0 \ar[r] & \cdots
}
\end{eqnarray*}
becomes a quasi-isomorphism after tensoring $\mathbb{Z}/m$ for any integer $m$. But $\prod_{p}\mathbb{Z}_p$ is not mild.\\
\end{remark}

To apply theorem \ref{prop-comparison1} we need the following corollary \ref{cor-dual2}.
\begin{lemma}\label{lem-dual1}
Let $\varphi:M\rightarrow N$ be a homomorphism between two $\mathbb{Z}/n$-modules. If $\varphi^{\vee}:N^\vee\rightarrow M^\vee$ is an isomorphism, then so is $\varphi$.
\end{lemma}
\emph{Proof}: First we show $\phi$ is injective. Since $\varphi^{\vee\vee}:M^{\vee\vee}\rightarrow N^{\vee\vee}$ is an isomorphism, we need only show that the natural homomorphism $M\rightarrow M^{\vee\vee}$ is injective, then the injectivity of $\varphi$ follows from the commutative diagram
\begin{eqnarray*}
\xymatrix{
M \ar[d] \ar[r] & M^{\vee\vee} \ar[d]\\
N  \ar[r] & N^{\vee\vee}.
}
\end{eqnarray*}
As an abelian group, $M$ is bounded (i.e., all elements have bounded finite orders), so by Baer-Pr\"{u}fer theorem \cite[4.3.5]{Rob96} it is a direct sum of cyclic groups of the form $\mathbb{Z}/n_1$ where $n_1$ is a factor of $n$. Thus for every nonzero element $x$ of $M$, there is a nonzero homomorphism $\alpha: M\rightarrow \mathbb{Z}/n$ such that $\alpha(x)\neq 0$. This means $M\rightarrow M^{\vee\vee}$ is injective.\\
\indent Next we show $\varphi$ is surjective. Let $L$ be the cokernel of $\varphi$. Since $\mathbb{Z}/n$ is an injective $\mathbb{Z}/n$-module, we have an exact sequence
\begin{eqnarray*}
0\rightarrow L^\vee\rightarrow N^\vee\rightarrow M^\vee.
\end{eqnarray*}
So $L^\vee=0$, thus $L^{\vee\vee}=0$. But we have shown that $L\rightarrow L^{\vee\vee}$ is an injective homomorphism, so $L=0$.\pqed

\begin{corollary}\label{cor-dual2}
Let $K_\bullet$ and $L_\bullet$ be two chain complexes of free abelian groups, and $\Phi: K_\bullet\rightarrow L_\bullet$ a homomorphism of chain complexes. Let $n$ be an integer. Suppose that the induced homomorphism $\Hom^\bullet(L_\bullet, \mathbb{Z}/n)\rightarrow \Hom^\bullet(K_\bullet, \mathbb{Z}/n)$ is a quasi-isomorphism. Then $\Phi\otimes \mathbb{Z}/n: K_\bullet\otimes \mathbb{Z}/n\rightarrow L_\bullet\otimes \mathbb{Z}/n$ is also  a quasi-isomorphism.
\end{corollary}
\emph{Proof}: For an abelian group $A$, there is a natural isomorphism $\Hom_\mathbb{Z}(A,\mathbb{Z}/n)\cong \Hom_{\mathbb{Z}/n}(A\otimes \mathbb{Z}/n,\mathbb{Z}/n)$. Thus $\Hom^\bullet(L_\bullet, \mathbb{Z}/n)\rightarrow \Hom^\bullet(K_\bullet, \mathbb{Z}/n)$ is the dual of $\Phi\otimes \mathbb{Z}/n$ as complexes of $\mathbb{Z}/n$-modules. Then by the  universal coefficient theorem (\ref{uct2}) and the fact that $\mathbb{Z}/n$ is an injective $\mathbb{Z}/n$-module, we have a commutative diagram of natural isomorphisms 
\begin{eqnarray*}
\xymatrix{
H^{i}(\Hom(L_\bullet, \mathbb{Z}/n)) \ar[d]^{\cong} \ar[r]^{\cong} & H_i (L_\bullet\otimes \mathbb{Z}/n)^\vee \ar[d]^{\cong}\\
H^{i}(\Hom(K_\bullet, \mathbb{Z}/n)) \ar[r]^{\cong}  & H_i (K_\bullet\otimes \mathbb{Z}/n)^\vee.}
\end{eqnarray*}
Then the conclusion follows from proposition \ref{lem-dual1}.\pqed

\section{Mildness of Suslin homology over an algebraically closed subfield of \texorpdfstring{$\mathbb{C}$}{}}

\subsection{Presheaves of relative zero cycles and Suslin homology}
For a field $\Bbbk$, we introduce several category of schemes over $\Bbbk$:
\begin{enumerate}
   \item[(i)] $\mathbf{Sch}/\Bbbk$ is the category of noetherian schemes separated over $\Bbbk$;
   \item[(ii)] $\mathbf{Nor}/\Bbbk$ is the full subcategory of $\mathbf{Sch}/\Bbbk$ consisting of normal noetherian schemes separated over $\Bbbk$;
   \item[(iii)] $\mathbf{Sch}_{\mathrm{ft}}/\Bbbk$ is the  full subcategory of $\mathbf{Sch}/\Bbbk$  consisting of schemes separated and of finite type over $\Bbbk$;
   \item[(iv)] $\mathbf{Nor}_{\mathrm{ft}}/\Bbbk$ is the full subcategory of $\mathbf{Sch}_{\mathrm{ft}}/\Bbbk$ consisting of normal noetherian schemes separated and of finite type over $\Bbbk$.
 \end{enumerate}
  We call $\Delta_{\Bbbk}^i:=\Spec(\Bbbk[t_0,\cdots,t_i]/(t_0+\cdots+t_i-1))$ the algebraic $i$-simplex over $\Bbbk$. When the base field $\Bbbk$ is obvious in the context, we omit the subscript and denote the $i$-th simplex by $\Delta^i$. For $X\in\mathrm{ob}(\mathbf{Sch}/\Bbbk)$, denote by $C_i(X)$ the  abelian group freely generated by the closed integral subschemes of $\Delta^i\times X$ which are surjective and finite over $\Delta^i$, and $C_i^{\mathrm{eff}}(X)$ the abelian sub-monoid of $C_i(X)$ consisting of the effective cycles. More generally, for any  $U\in \mathbf{Nor}/\Bbbk$, denote $z_0^c(X)(U)$ to be the abelian group freely generated by the integral subschemes of $U\times X$ which are finite and surjective on a component of $U$.

\begin{definition}\label{def-Suslinhomology}
The \emph{algebraic singular homology} of $X$, or the \emph{Suslin homology} of $X$, is defined to be the homotopy group of the simplicial abelian group $C_\bullet(X)$, i.e.,
\begin{eqnarray*}
H_i^{S}(X):=\pi_i(C_\bullet(X)).\
\end{eqnarray*}
For a simplicial abelian group $C_\bullet$, denote by $K(C_\bullet)$  the associated complex of abelian groups. Then for an abelian group $A$, the Suslin homology of $X$ with coefficients in $A$ is $H_i(K(C_\bullet(X))\otimes A)$. The \emph{Suslin cohomology} of $X$ is defined as
\[
H_S^i(X,A):=H^i\big(\Hom(K(C_\bullet(X)),A)\big).
\]
\end{definition}

Recall that a morphism of schemes is called a \emph{qfh covering} if it is a \emph{universal topological epimorphism} and is quasi-finite. In particular,  finite surjective morphisms and  étale surjective morphisms of finite type are qfh coverings. For more on this notion, see \cite[section 10]{SV96}, \cite[section 4]{SV00} and \cite{Voe96}. To get a better understanding of Suslin homology, and for later use, let us recall several presheaves of cycles on $\mathbf{Sch}/\Bbbk$ or $\mathbf{Sch}_{\mathrm{ft}}/\Bbbk$:
\begin{enumerate}
  \item[(i)] Restricting to $\mathbf{Nor}_{\mathrm{ft}}/\Bbbk$, $z_0^c(X)$ is a qfh-sheaf, and extends to a qfh-sheaf on $\mathbf{Sch}_{\mathrm{ft}}/\Bbbk$ (\cite[section 6]{SV96});
  \item[(ii)] $\mathbb{Z}(X)$ is the presheaf on $\mathbf{Sch}/\Bbbk$ or $\mathbf{Sch}_{\mathrm{ft}}/\Bbbk$, such that $\mathbb{Z}(X)(U)$ is the abelian group freely generated by the $\Bbbk$-morphisms from $U$ to $X$.  $\mathbb{Z}(X)_{\mathrm{qfh}}$ is the qfh-sheafification of  $\mathbb{Z}(X)$;
  \item[(iii)] the presheaf $c_{\mathrm{equi}}(X,0)$ on $\mathbf{Sch}/\Bbbk$ is introduced in \cite[section 3]{SV00}.
\end{enumerate}
These (pre-)sheaves have the following properties.
\begin{proposition}\phantomsection\label{prop-presheaves-cycles}
\begin{enumerate}
  \item[(i)] There is a natural isomorphism of qfh-sheaves $z_0^c(X)\cong \mathbb{Z}[\frac{1}{p}](X)_{\mathrm{qfh}}$ on $\mathbf{Sch}_{\mathrm{ft}}/\Bbbk$, where $p=\mathrm{char}(\Bbbk)$. 
  \item[(ii)] For regular  schemes $U\in \mathbf{Sch}/\Bbbk$, $c_{\mathrm{equi}}(X,0)(U)\cong z_0^c(X)(U)$.  
  \item[(iii)] The qfh-sheafification of $c_{\mathrm{equi}}(X,0)$ is canonically isomorphic to $\mathbb{Z}(X)_{\mathrm{qfh}}$.
\end{enumerate}
\end{proposition}

In fact, (i) is \cite[theorem 6.7]{SV96}, (ii) is \cite[corollary 3.4.6]{SV00}, and (iii) is \cite[theorem 4.2.12]{SV00}.

Let $X\in \mathbf{Sch}/\Bbbk$ satisfying that any finite subset of $X$ is contained in an affine open subset, for example quasiprojective schemes over $\Bbbk$. Thus  the $j$-th symmetric product $\mathrm{S}^j(X)$ exists as a scheme. By \cite[theorem 6.8]{SV96},  for any normal connected scheme $S$  in $\mathbf{Sch}/\Bbbk$, there is a natural isomorphism
\begin{eqnarray}\label{eq-sv96.6.8-0}
\Hom(S, \coprod_{j=0}^{\infty}\mathrm{S}^j(X))[\frac{1}{p}]\cong z_0^c(X)^{\mathrm{eff}}(S),
\end{eqnarray}
where $p=\mathrm{char}(\Bbbk)$. In the following of this subsection we generalize this result to all $X\in \mathbf{Sch}/\Bbbk$. 

\begin{definition}
Let $S$ be a scheme. 
Denote by $\mathbf{Es}/S$ the category of algebraic spaces separated and of finite type over $S$. A representable morphism $f:X\rightarrow Y$ in $\mathbf{Es}/S$ is called a \emph{qfh covering} if for an étale covering $U\rightarrow Y$ by a scheme $U$ of finite type over $S$, $f_U: X\times_Y U\rightarrow U$ is qfh. More generally, a morphism  $f:X\rightarrow Y$ in $\mathbf{Es}/S$ is called a \emph{qfh covering} if for an étale covering $g:V\rightarrow X$ by a scheme $V$ of finite type over $S$, $f\circ g: V\rightarrow Y$ is qfh. The independence of the choices of $U$ and $V$ is obvious.
\end{definition}

\begin{lemma}\label{lem-qfh-es}
Let $S$ be a scheme and $\mathscr{F}$ a qfh-sheaf over $\mathbf{Sch}/S$. For every $S$-algebraic space $X$, choosing an étale cover $U\rightarrow X$ where $U$ is an $S$-scheme, put
\[
\mathscr{F}(X)=\ker(\mathscr{F}(U)\rightrightarrows \mathscr{F}(U\times_X U)).
\]
Then $\mathscr{F}$ becomes a qfh-sheaf over $\mathbf{Es}/S$.
\end{lemma}
Proof: We omit the standard proof; the only fact we need is that  étale coverings in $\mathbf{Sch}/S$ and $\mathbf{Es}/S$ are  qfh coverings.\pqed

\begin{lemma}\label{lem-sp-es}
Let $S$ be a locally noetherian scheme. For any scheme $X$ separated and locally of finite type over $S$, the quotient stack 
\[
\mathscr{S}^d(X/S)=\underbrace{X\times_S X\times_S\cdots \times_S X}_{d\mbox{\footnotesize{-factors of}}\ X}/\mathfrak{S}_d
 \] 
has a coarse moduli, which is  a algebraic space separated and locally of finite type over $S$, and we denote it by $\mathrm{S}^d(X/S)$. It has the following properties:
\begin{enumerate}
  \item[(i)] If $X$ is  finite and surjective over $S$, $\mathrm{S}^d(X/S)$ is  a scheme and coincides with the construction of \cite[p. 81]{SV96}. 
  \item[(ii)] $\mathrm{S}^d(X/S)$ is covariant in $X$, i.e., for an $S$-morphism of two separated $S$-schemes $X\rightarrow Y$, there is an naturally associated morphism of algebraic spaces $\mathrm{S}^d(X/S)\rightarrow \mathrm{S}^d(Y/S)$.
  \item[(iii)] The formation of $\mathrm{S}^d(X/S)$ commutes with flat base changes, i.e., $\mathrm{S}^d(X\times_S T/T)$ is naturally isomorphic to $\mathrm{S}^d(X/S)\times_S T$ if $T\rightarrow S$ is a flat morphism.
\end{enumerate}
\end{lemma}
Proof: By the assumptions on $X$ and $S$, the algebraic stack $\mathscr{S}^d(X/S)$ is locally of finite presentation with finite diagonal over $S$. So the existence of the coarse moduli and the properties (ii)-(iii) follows from the Keel-Mori theorem, see e.g. \cite[theorem 11.1.2]{Ols16}. In the case of (i), the coarse moduli coincides with the categorical quotient $\mathrm{S}^d(X/S)$.\pqed

For $X\in \mathbf{Sch}/\Bbbk$, we denote the $i$-th symmetric product of $X$ by $\mathrm{S}^i(X)$, which is an algebraic space by lemma \ref{lem-sp-es}.

\begin{proposition}\label{prop-iso-symprod}
Suppose $\mathrm{char}(\Bbbk)=p$. Let $X$ be separated scheme of finite type over $\Bbbk$.  Then for any  $S\in \mathbf{Nor}/\Bbbk$, there is  a natural isomorphism
\begin{eqnarray}\label{eq-sv96.6.8-1}
\Hom(S, \coprod_{j=0}^{\infty}\mathrm{S}^j(X))[\frac{1}{p}]\cong z_0^c(X)^{\mathrm{eff}}(S).
\end{eqnarray}
In particular, there is  a natural isomorphism
\begin{eqnarray}\label{eq-sv96.6.8}
\Hom(\Delta^i, \coprod_{j=0}^{\infty}\mathrm{S}^j(X))[\frac{1}{p}]\cong C_i^{\mathrm{eff}}(X).
\end{eqnarray}
\end{proposition}
\emph{Proof:}  Using lemmas \ref{lem-qfh-es} and \ref{lem-sp-es} one can check that the proof of \cite[theorem 6.8]{SV96} (see \cite[p. 81 line 10 to p. 82]{SV96}) still works. \pqed 


\subsection{On the Dold-Thom  theorem}
The goal of this subsection is to supply a proof for \cite[theorem 8.2]{SV96}(=corollary \ref{cor-kan-fibration}), which was stated without proof as a variant of Dold-Thom theorem \cite{DT58}. We begin by recalling some definitions and notations of loc. cit.
\begin{definition}\label{def-SP}
Let $(\mathsf{X},\mathsf{x}_0)$ be a pointed topological space, i.e.  a topological space $\mathsf{X}$  with a distinguished point $\mathsf{x}_0\in \mathsf{X}$. Denote the symmetric product of $(\mathsf{X},\mathsf{x}_0)$ by $\mathrm{S}^i(\mathsf{X},\mathsf{x}_0)$, for $i\geq 0$, which is still pointed ($\mathrm{S}^0(\mathsf{X},\mathsf{x}_0)$ is the single point space $\{\mathsf{x}_0\}$). There are natural inclusions
\begin{equation}\label{eq-def-SP}
\mathrm{S}^i(\mathsf{X},\mathsf{x}_0)\hookrightarrow \mathrm{S}^{i+1}(\mathsf{X},\mathsf{x}_0),\quad 
(x_1,...,x_i)\mapsto (x_0,x_1,...,x_i).
\end{equation}
The colimit $\varinjlim_i \mathrm{S}^i(\mathsf{X},\mathsf{x}_0)$ is denoted by $\mathrm{SP}(\mathsf{X},\mathsf{x}_0)$, which is a topological abelian monoid in an obvious way.
\end{definition}

\begin{definition}
Let $(\mathsf{X},\mathsf{{x}_0)}$ be a pointed topological space. 
Denote by $\tau$ the map $\mathsf{X}\vee \mathsf{X}\rightarrow \mathsf{X}\vee \mathsf{X}$ which exchange the two copies of $\mathsf{X}$. We introduce a relation 
\[
\mathsf{x}\sim \mathsf{x}+\mathsf{x}'+\tau(\mathsf{x}').
\]
Then we denote the quotient space $\mathrm{SP}(\mathsf{X}\vee \mathsf{X},\mathsf{x}_0)/\sim$ by $\mathrm{AG}(\mathsf{X},\mathsf{x}_0)$, which is a commutative topological group, with the identity element  $\mathsf{x}_0$. Denote the two canonical embeddings $\mathrm{SP}(\mathsf{X},\mathsf{x}_0)\hookrightarrow \mathrm{SP}(\mathsf{X}\vee \mathsf{X},\mathsf{x}_0)/\sim=\mathrm{AG}(\mathsf{X},\mathsf{x}_0)$ by $\iota_1$ and $\iota_2$.
\end{definition}

Recall the following theorem of Dold and Thom \cite[Satz 6.10]{DT58}.
\begin{theorem}\label{thm-dt58}
Suppose $q>0$. 
\begin{enumerate}
  \item[(i)] Let $(\mathsf{X},\mathsf{x}_0)$ be a pointed connected  CW-complex. Then there are natural isomorphisms
  \begin{equation}\label{eq-dt58-sp}
  H_q(\mathsf{X},\mathbb{Z})\xrightarrow{\sim} \pi_q(\mathrm{SP}(\mathsf{X},\mathsf{x}_0)).
  \end{equation}
  \item[(ii)] Let $(\mathsf{X},\mathsf{x}_0)$ be a pointed countable connected  simplicial complex. Then there are  natural isomorphisms
  \begin{equation}\label{eq-dt58-ag}
  H_q(\mathsf{X},\mathbb{Z})\xrightarrow{\sim} \pi_q(\mathrm{AG}(\mathsf{X},\mathsf{x}_0)).
  \end{equation}
\end{enumerate}
\end{theorem}

We denote by $\mathsf{\Delta}^i$ the topological $i$-simplex, and by $\mathrm{Sin}(\mathsf{X})$ the simplicial set $\mathrm{Hom}(\mathsf{\Delta}^{\bullet},\mathsf{X})$ of continuous maps of simplices to a topological space $\mathsf{X}$.

\begin{definition}
Let $(\mathsf{X},\mathsf{x}_0)$ be a pointed topological space. The projection $\coprod_{j=0}^{\infty}\mathrm{S}^j(\mathsf{X})\rightarrow \mathrm{SP}(\mathsf{X},\mathsf{x}_0)$ induces a map of simplicial abelian monoids
\[
\Upsilon: \mathrm{Sin}\big(\coprod_{j=0}^{\infty}\mathrm{S}^j(\mathsf{X})\big)\rightarrow \mathrm{Sin}\big(\mathrm{SP}(\mathsf{X},\mathsf{x}_0)\big)
\]
which extends to a map of simplicial abelian groups
\[
\Upsilon^{+}: \mathrm{Sin}\big(\coprod_{j=0}^{\infty}\mathrm{S}^j(\mathsf{X})\big)^{+}\rightarrow \mathrm{Sin}\big(\mathrm{SP}(\mathsf{X},\mathsf{x}_0)\big)^{+}.
\]  
\end{definition}
We have a commutative diagram of simplical monoids and simplical abelian groups
\begin{equation*}
\xymatrix{
  \mathrm{Sin}\big(\coprod_{j=0}^{\infty}S^j(\mathsf{X})\big) \ar[d]_{\Upsilon} \ar[r] &
  \mathrm{Sin}\big(\coprod_{j=0}^{\infty}S^j(\mathsf{X})\big)^{+} \ar[d]^{\Upsilon^{+}} & \\
  \mathrm{Sin}\big(\mathrm{SP}(\mathsf{X},\mathsf{x}_0)\big) \ar[r]^{\Theta} & \mathrm{Sin}\big(\mathrm{SP}(\mathsf{X},\mathsf{x}_0)\big)^{+} \ar[r]^{\Xi} & 
  \mathrm{Sin}\big(\mathrm{AG}(\mathsf{X},\mathsf{x}_0)\big).
}
\end{equation*}
By \cite[1.5, 17.1]{May67}, all the above simplicial sets are Kan complexes.

\begin{proposition}\label{prop-surj-monoid}
Let $G$ be a path connected  topological abelian monoid with unit element $e$. Then the natural  maps of homotopy groups
\[
\pi_n(\mathrm{Sin}(G,e))\rightarrow \pi_n(\mathrm{Sin}(G,e)^{+})
\] induced by the map from the simplical monoid $\mathrm{Sin}(G,e)$ to its simplicial group completion $\mathrm{Sin}(G,e)^{+}$, are surjective for $n\geq 0$.
\end{proposition}
Proof: Suppose  $\alpha\in \mathrm{Sin}_n(G,e)^{+}$  such that $\partial_i \alpha=0$ for $0\leq i\leq n$. There are continuous maps $\beta:\mathsf{\Delta}^n\rightarrow G$ and $\gamma:\mathsf{\Delta}^n\rightarrow G$ such that $\alpha=[\beta]-[\gamma]$,  thus $\partial_i \beta=\partial_i \gamma$ as maps $\mathsf{\Delta}^{n-1}\rightarrow G$. We need to show that there is a map $\sigma:\mathsf{\Delta}^n\rightarrow G$ such that $\partial_i \sigma$ equals the constant map $\mathsf{\Delta}^{n-1}\rightarrow e$ for $0\leq i\leq n$, and $\sigma$ is homotopic to $\beta- \gamma $ in  $\mathrm{Sin}(G,e)^{+}$. For convenience we write $\beta$ and $\gamma$ as maps from the $n$-cube $[0,1]^n$ to $G$. It is clear that there are maps 
\[
\beta': [0,1]^n\rightarrow G
\]
and  
\[
\gamma':[0,1]^n\rightarrow G
\]
such that $\beta- \gamma$ is boundary-preserving homotopic to $\beta'-\gamma'$ in $\mathrm{Sin}(G,e)^{+}$ and 
\begin{equation}\label{eq-surj-monoid-1}
\beta'|_{[0,1]^{n-1}\times{(1,0)}}=\gamma'_{[0,1]^{n-1}\times{(1,0)}}:[0,1]^{n-1}\rightarrow G,
\end{equation}
and 
\[
\beta'|_{\partial([0,1]^n)\backslash [0,1]^{n-1}\times{(1,0)}}=\gamma'|_{\partial([0,1]^n)\backslash [0,1]^{n-1}\times{(1,0)}}=x_0
\]
where $x_0$ is a point in $G$ that lies in the image of (\ref{eq-surj-monoid-1}).  Since $G$ is path connected, choosing a path from $x_0$ to $e$ we can find $\beta''$ and $\gamma''$ such that $\beta''-\gamma''$ is boundary-preserving homotopic to $\beta'-\gamma'$ in $\mathrm{Sin}(G,e)^{+}$, $\beta''$ and $\gamma''$ satisfy (\ref{eq-surj-monoid-1}), and 
\[
\beta''|_{\partial([0,1]^n)\backslash [0,1]^{n-1}\times{(1,0)}}=\gamma''|_{\partial([0,1]^n)\backslash [0,1]^{n-1}\times{(1,0)}}=e.
\]
So without loss of generality we assume $x_0=e$ and still consider $\beta'$ and $\gamma'$. We denote the map (\ref{eq-surj-monoid-1}) by $\tau:[0,1]^{n-1}\rightarrow G$. Writing a point of $[0,1]^n$ as $(x,y)$, where $x\in [0,1]^{n-1}$ and $y\in [0,1]$, we set
\begin{equation}\label{eq-surj-monoid-2}
\sigma(x,y)=\begin{cases}
\beta'(x,2y), & 0\leq y\leq \frac{1}{2},\\
\gamma'(x,2-2y), & \frac{1}{2}\leq y\leq 1.
\end{cases}
\end{equation}
Then $\sigma|_{\partial [0,1]^n}=e$. It suffices to show that $\sigma-(\beta'-\gamma')$ is homotopic to 0 in $\mathrm{Sin}(G,e)^{+}$. It is clear that $\gamma'$ is boundary-preserving homotopic to the map 
\begin{equation}\label{eq-surj-monoid-3}
\gamma'''(x,y)=\begin{cases}
e, & 0\leq y\leq \frac{3}{4},\\
\gamma'(x,4y-3), & \frac{3}{4}\leq y\leq 1,
\end{cases}
\end{equation}
and $\sigma$ is boundary-preserving homotopic to the map
\begin{equation}\label{eq-surj-monoid-4}
\sigma'(x,y)=\begin{cases}
\beta'(x,2y), & 0\leq y\leq \frac{1}{2},\\
\gamma'(x,3-4y), & \frac{1}{2}\leq y\leq \frac{3}{4},\\
e, & \frac{3}{4}\leq y\leq 1.
\end{cases}
\end{equation}
Then $\sigma+\gamma'$ is  boundary-preserving homotopic to
\begin{equation}\label{eq-surj-monoid-5}
(\sigma'+\gamma''')(x,y)=\begin{cases}
\beta'(x,2y), & 0\leq y\leq \frac{1}{2},\\
\gamma'(x,3-4y), & \frac{1}{2}\leq y\leq \frac{3}{4},\\
\gamma'(x,4y-3), & \frac{3}{4}\leq y\leq 1.
\end{cases}
\end{equation}
We define a homotopy $F: [0,1]^n\times[0,1]$ by
\begin{equation}\label{eq-surj-monoid-6}
F(x,y,t)=\begin{cases}
\beta'(x,2y), & 0\leq y\leq \frac{1}{2},\\
\gamma'(x,1+t(2-4y)), & \frac{1}{2}\leq y\leq \frac{3}{4},\\
\gamma'(x,1+t(4y-4), & \frac{3}{4}\leq y\leq 1.
\end{cases}
\end{equation}
Then $F(x,y,t)$ is a boundary preserving homotopy from $\sigma'+\gamma'''$ to
\begin{equation*}
F(x,y,0)=\begin{cases}
\beta'(x,2y), & 0\leq y\leq \frac{1}{2},\\
\tau(x), & \frac{1}{2}\leq y\leq 1,
\end{cases}
\end{equation*}
which in turn is boundary preserving homotopic to $\beta'$. Thus we are done.\pqed

\begin{proposition}\label{prop-we}
Suppose $\mathsf{X}$ is path connected and admits a countable triangulation. Then $\Theta$ is a weak equivalence.
\end{proposition}
Proof: By \cite[Satz 6.10 III]{DT58}, $\Xi\circ \Theta$ is a weak equivalence, so 
\begin{equation}\label{eq-prop-we}
\pi_i\big(\mathrm{Sin}\big(\mathrm{SP}(\mathsf{X},\mathsf{x}_0)\big)\big)\rightarrow \pi_i\big(\mathrm{Sin}\big(\mathrm{SP}(\mathsf{X},\mathsf{x}_0)\big)^{+}\big)
\end{equation}
 is injective for any $i$. Since $\mathrm{SP}(\mathsf{X},\mathsf{x}_0)$ is a topological abelian monoid and is path connected,  (\ref{eq-prop-we}) is also surjective by proposition \ref{prop-surj-monoid}. \pqed

\begin{proposition}\label{prop-kan-fib}
Suppose $\mathsf{X}$ is path connected. Then  $\Upsilon^{+}$ is a Kan fibration, and there is a fiber sequence of Kan pairs
\begin{equation}\label{eq-fib-seq}
(\mathbb{Z},0)\rightarrow \Big(\mathrm{Sin}\big(\coprod_{j=0}^{\infty}S^j(\mathsf{X})\big)^{+},0\Big) \rightarrow \Big(\mathrm{Sin}\big(\mathrm{SP}(\mathsf{X},\mathsf{x}_0)\big)^{+},0\Big),
\end{equation}
where $\mathbb{Z}$ is, by an abuse of notation, the simplical abelian group generated by $\mathbb{Z}$, which is isomorphic to $\mathrm{Sin}(\mathbb{Z})$ if we view $\mathbb{Z}$ as a discrete topological group.
\end{proposition}
Proof: Suppose $\beta_0,...,\beta_{k-1},\beta_{k+1},...,\beta_{n+1}\in \Hom(\mathsf{\Delta}^{n},\coprod_{j=0}^{\infty}S^j(\mathsf{X}))^{+}$ satisfying 
\begin{equation}\label{eq-kan-fib-1}
\partial_{i}\beta_j=\partial_{j-1}\beta_i,\ \mbox{for}\ i<j,\ i\neq k,\ j\neq k,
\end{equation}
and $\alpha\in \Hom(\mathsf{\Delta}^{n+1},\mathrm{SP}(\mathsf{X},\mathsf{x}_0))^{+}$ such that
\begin{equation}\label{eq-kan-fib-2}
\partial_{i}\alpha=\Upsilon^{+}(\beta_i),\ \mbox{for}\ i\neq k.
\end{equation}
 Let $\beta_{i}=\beta_{i,1}-\beta_{i,2}$, where $\beta_{i,1}\in \Hom(\mathsf{\Delta}^{n},S^{j_{i,1}}(\mathsf{X}))$, $\beta_{i,2}\in \Hom(\mathsf{\Delta}^{n},S^{j_{i,2}}(\mathsf{X}))$,  for $0\leq i\leq n+1$, $i\neq k$. Then (\ref{eq-kan-fib-1}) implies that $j_{i,1}-j_{i,2}$ are equal for all $0\leq i\leq n+1$, $i\neq k$. We denote this common integer $j_{i,1}-j_{i,2}$ by $t$. Suppose $\alpha=\alpha_1- \alpha_2$, where $\alpha_1, \alpha_2\in \Hom(\mathsf{\Delta}^{n+1},\mathrm{SP}(\mathsf{X},\mathsf{x}_0))$. Since $\mathsf{\Delta}^{n+1}$ is compact, there exists $m\in \mathbb{N}$ such that $\alpha_i$ factors as 
\[
\mathsf{\Delta}^{n+1}\xrightarrow{\tilde{\alpha}_i} S^m(\mathsf{X}) \hookrightarrow \mathrm{SP}(\mathsf{X},\mathsf{x}_0)
\]
for $i=1,2$. Let $\tilde{\beta}=\tilde{\alpha}_1-\tilde{\alpha}_2\in \Hom(\mathsf{\Delta}_{n},S^m(\mathsf{X}))^{+}$. By (\ref{eq-kan-fib-2}), one has 
\[
\partial_i \tilde{\beta}+(\mathsf{\Delta}^n\rightarrow |t|\mathsf{x}_0)=\beta_i
\]
or 
\[
\partial_i \tilde{\beta}=\beta_i+\big(\mathsf{\Delta}^n\rightarrow (-|t|)\mathsf{x}_0\big),
\]
depending on $t$ is positive or negative. We set $\beta=\tilde{\beta}+(\mathsf{\Delta}^{n+1}\rightarrow |t|\mathsf{x}_0)$ or $\beta=\tilde{\beta}-\big(\mathsf{\Delta}^{n+1}\rightarrow (-|t|)\mathsf{x}_0\big)$ respectively. 
Then $\beta\in \Hom(\mathsf{\Delta}^{n+1},\coprod_{j=0}^{\infty}S^j(\mathsf{X}))^{+}$ and  $\partial_i \beta=\beta_i$ for $i\neq k$, and $\Upsilon^+(\beta)=\alpha$. This proves the first statement.

Taking $\beta_i=0$ for $0\leq i\leq n+1$, $i\neq k$, and $\alpha=0$, in the above proof, we obtain that the preimage $\Upsilon^{+}$ consists of the classes represented by the constant maps $\mathsf{\Delta}^n\rightarrow 0=S^0(\mathsf{X})$. So we have the fiber sequence (\ref{eq-fib-seq}). \pqed

\begin{corollary}\label{cor-kan-fibration}
Let $\mathsf{X}$ be the underlying space of a countable simplicial complex. Then there are natural isomorphisms 
\begin{equation}
\pi_i\Big(\mathrm{Sin}\big(\coprod_{j=0}^{\infty}S^j(\mathsf{X})\big)^{+},0\Big)\cong H_i^{\mathrm{sing}}(\mathsf{X},\mathbb{Z}). 
\end{equation}
\end{corollary}
Proof: Without loss of generality we can assume that $\mathsf{X}$ is path connected.
 The fiber sequence (\ref{eq-fib-seq}) induces a long exact sequence of homotopy groups
 \[
  ...\rightarrow \pi_i(\mathbb{Z},0)\rightarrow 
  \pi_i\Big(\mathrm{Sin}\big(\coprod_{j=0}^{\infty}S^j(\mathsf{X})\big)^{+},0\Big)\rightarrow 
  \pi_i\Big(\big(\mathrm{SP}(\mathsf{X},\mathsf{x}_0)\big)^{+},0\Big)\rightarrow 
  \pi_{i-1}(\mathbb{Z},0)\rightarrow ...
  \] 
  Since $\pi_i(\mathbb{Z},0)=0$ for $i>0$ and $\pi_0(\mathbb{Z},0)=\mathbb{Z}$, by (\ref{thm-dt58}), it suffices to show that the connecting homomorphism
 \begin{equation}\label{eq-connection-map}
  \partial_{\sharp}: \pi_1\Big(\big(\mathrm{SP}(\mathsf{X},\mathsf{x}_0)\big)^{+},0\Big)\rightarrow \pi_0(\mathbb{Z},0)
 \end{equation}
  is zero. Let $\alpha\in \Hom(\mathsf{\Delta}^{1},\mathrm{SP}(\mathsf{X},\mathsf{x}_0))^{+}$ such that $\partial_0 \alpha=\partial_1 \alpha=0$. By the construction of $\partial_{\sharp}$ \cite[p. 27]{May67}, we need  to find $\beta\in \big(\Hom(\mathsf{\Delta}^{1},\coprod_{j=0}^{\infty}\mathrm{S}^j(\mathsf{X}))^{+},0\big)$ such that $\Upsilon^{+}(\beta)=\alpha$ and $\partial_1 \beta=0$, then $\partial_0 \beta=\partial_{\sharp}\alpha$. Suppose $\alpha=\alpha_1- \alpha_2$, where $\alpha_1, \alpha_2\in \Hom(\mathsf{\Delta}^{1},\mathrm{SP}(\mathsf{X},\mathsf{x}_0))$.   By the argument as the proof of proposition \ref{prop-kan-fib}, we can find $\tilde{\alpha}_{1},\tilde{\alpha}_2\in \Hom(\mathsf{\Delta}^{1},\mathrm{S}^m(\mathsf{X}))$ such that $\Upsilon^{+}(\tilde{\alpha}_{1}-\tilde{\alpha}_2)=\alpha_1- \alpha_2$. Let $\beta=\tilde{\alpha}_{1}-\tilde{\alpha}_2$, then $\partial_0 \beta=\partial_1 \beta=0$, hence $\partial_{\sharp}\alpha=0$. \pqed

Using corollary \ref{cor-kan-fibration}, the following theorem was shown in the proof of \cite[theorem 8.3]{SV96}.
\begin{theorem}\label{thm-SV8.3}
For every $X\in \mathbf{Sch}/\mathbb{C}$ and every $n\in\mathbb{Z}$, the homomorphism of complexes
\[
K(\Hom(\Delta^\bullet, \coprod_{j=0}^{\infty}\mathrm{S}^j(X))^+) \rightarrow  K(\mathrm{Sin}_\bullet(\coprod_{i=0}^\infty \mathrm{S}^i (X(\mathbb{C})))^+)
\]
 becomes a quasi-isomorphism after applying $\otimes^{\mathbf{L}} \mathbb{Z}/n$. 
\end{theorem}
For the reader's convenience we sketch the proof of \cite[theorem 8.3]{SV96}. In loc. cit,  a commutative diagram is constructed:
\[
\xymatrix{
  j^*\big(C_*(Z)\big)\ar[r] \ar[d] & j^*(\mathscr{F}_*)^{\sim} \ar[d] & j^*(\mathscr{M}^+)  \ar[l] \ar@{=}[r] & \mathscr{M}_{\mathrm{top}}^{+} \ar[d]^{=} \\
  C_*(\mathscr{F}_{\mathrm{top}}) \ar[r] & \big((\mathscr{F}_{\mathrm{top}})_*\big)^{\sim} &
  (\mathscr{F}_{\mathrm{top}})^{\sim} \ar[l] \ar[l] \ar@{=}[r] & \mathscr{M}_{\mathrm{top}}^{+}.
}
\]
By \cite[theorem 7.6]{SV96}, applying $\Ext(\ \cdot\ ,\mathbb{Z}/n)$ to the upper row one obtain quasi-isomorphisms. These quasi-isomorphisms follow from the rigidity theorem \cite[theorem 4.4]{SV96}. Similarly the arrows in the lower row are quasi-isomorphisms, by lemma \ref{lem-rigidity-top}. Then the conclusion follows.\pqed

\begin{lemma}\label{lem-rigidity-top}
Let $\mathscr{F}$ be a homotopy invariant  presheaf on the category $\mathbf{TriTop}$ of triangulable topological spaces (i.e., the topological spaces which are underlying spaces of simplicial complexes). Then for any $\mathsf{X}\in \mathbf{TriTop}$ and a point $\mathsf{x}\in \mathsf{X}$, 
\[
\mathscr{F}_{\mathsf{x}}=\mathscr{F}(\mathrm{pt}).
\]
\end{lemma}
Proof: Since $\mathsf{X}\in \mathbf{TriTop}$ is locally homotopy equivalent to a point, the conclusion follows from the homotopy invariance of $\mathscr{F}$.\pqed

\subsection{Mildness of Suslin homology over an algebraically closed subfield of \texorpdfstring{$\mathbb{C}$}{}}

Now we are ready to show the mildness of Suslin homology groups for schemes over an algebraically closed subfield of $\mathbb{C}$, and the lifting properties (M1), (M2) for the homomorphisms.
\begin{theorem}\label{thm-mild-1}
Let $\Bbbk$ be an algebraically closed subfield of $\mathbb{C}$, $X$ a separated scheme of finite type over $\Bbbk$. Then $H_i^S(X,\mathbb{Z})$ are mild for $i\geq 0$, and
\begin{enumerate}
	\item[(i)] $s(H_{i-1}^S(X,\mathbb{Z}))+r(H_{i}^S(X,\mathbb{Z}))=r_{\mathbb{Q}_l}(H_{\et}^i(X,\mathbb{Q}_{\ell}))$, where $r_{\mathbb{Q}_{\ell}}$ means the rank of a $\mathbb{Q}_{\ell}$-vector space.
	\item[(ii)] There are canonical isomorphisms
	\[
		H_i^S(X,\mathbb{Z})_{\mathrm{tor}}\cong \bigoplus_{\ell}H^{i+1}_{\et}(X,\mathbb{Z}_{\ell})_{\mathrm{tor}}.
	\]
  \item[(iii)] Let $Y$ be a separated scheme of finite type over $\Bbbk$, $f:X\rightarrow Y$ a morphism of $\Bbbk$-schemes. Then the induced homomorphism $f_*: H_i^S(X,\mathbb{Z})\rightarrow H_i^S(Y,\mathbb{Z})$ satisfies the condition (M1).
  \item[(iv)] Let $g:X\rightarrow X$ be a $\Bbbk$-morphism. Then the induced endomorphism $g_*: H_i^S(X,\mathbb{Z})\rightarrow H_i^S(X,\mathbb{Z})$ satisfies the condition (M2).
\end{enumerate}
\end{theorem}
Proof:
When $\Bbbk=\mathbb{C}$, there is a natural embedding of cosimplicial spaces $\Delta_\mathrm{top}^\bullet\hookrightarrow \Delta^\bullet(\mathbb{C})$, which induces the bottom arrow in the following diagram
\begin{eqnarray}\label{graph1}
\xymatrix{
 K(C_\bullet(X)) \ar@{=}[d] \ar@{.>}[r]^{\Phi} & K(\mathbb{Z}(\mathrm{Sin}_{\bullet}(X(\mathbb{C}))) \ar[d] \\
K(\Hom(\Delta_\bullet, \coprod_{j=0}^{\infty}\mathrm{S}^j(X))^+) \ar[r] & K(\mathrm{Sin}_\bullet(\coprod_{i=0}^\infty \mathrm{S}^i (X(\mathbb{C})))^+).
}
\end{eqnarray}
The left identification is induced by (\ref{eq-sv96.6.8}). The right arrow is a free resolution by Dold-Thom.  The complexes $K(C_\bullet(X))$ and $K(\mathbb{Z}(\mathrm{Sin}_{\bullet}(X(\mathbb{C})))$ are complexes of free abelian groups, therefore there exists a lifting making the diagram commutative. We choose such a lifting and denote it by $\Phi$.
Apply proposition \ref{prop-comparison1} to $\Phi$. By theorem \ref{thm-SV8.3} $\Phi\otimes \mathbb{Z}/n$ are isomorphisms for all $n\in \mathbb{Z}$, hence we obtain the mildness of $H_i^S(X,\mathbb{Z})$ for $\Bbbk=\mathbb{C}$. 

For a subfield $\Bbbk$  of $\mathbb{C}$, 
there is a natural homomorphism $\Psi: K(C_\bullet(X))\rightarrow K(C_\bullet(X\times_{\Bbbk}\mathbb{C}))$. If  $\Bbbk$ is  algebraically closed, by \cite[theorem 7.8]{SV96}, there is a commutative diagram
\[
\xymatrix{
	H^i_S(X,\mathbb{Z}/n) \ar[d]^{\cong} \ar[r] & H^i_S(X\times_{\Bbbk}\mathbb{C},\mathbb{Z}/n) \ar[d]^{\cong}\\
	H^i_{\et}(X,\mathbb{Z}/n)  \ar[r] & H^i_S(X\times_{\Bbbk}\mathbb{C},\mathbb{Z}/n) 
}
\]
where both vertical arrows are isomorphisms. Since the bottom arrow is also an isomorphism by the smooth base change theorem, so is the top arrow. Then by corollary \ref{cor-dual2}, $\Psi\otimes \mathbb{Z}/n$ are isomorphisms for all $n\in \mathbb{Z}$. Now applying proposition \ref{prop-comparison1} to  the composition
\[
K(C_\bullet(X))\xrightarrow{\Psi} K(C_\bullet(X\times_{\Bbbk}\mathbb{C}))\xrightarrow{\Phi}K(\mathbb{Z}(\mathrm{Sin}_{\bullet}(X(\mathbb{C})))
\]
we obtain the mildness of $H_i^S(X,\mathbb{Z})$. Part (ii) follows from the universal coefficient theorem (\ref{eq-comparison1-2}) and the comparison theorem on étale cohomology and singular cohomology \cite[XVI, 4.1]{SGA4}. Parts (iii) and (iv) follow from proposition \ref{prop-comparison2} and \ref{prop-comparison3}.
\pqed

\section{Application of a rigidity theorem}
In theorem \ref{thm-mild-1} we assume that the base field can be imbedded into $\mathbb{C}$. In this section we are going to remove this restriction, using the Lefschetz principle. For this we need to know how the Suslin homology changes under an algebraically closed extension of the base field. A theorem of this type for Gillet's Chow groups appeared in \cite{Lec86}, and Jannsen \cite{Jan15} put it into a general framework. We begin by recalling a rigidity theorem of Jannsen.

\subsection{A rigidity theorem of Jannsen}
Let $\Bbbk$ be a field. Denote by $\mathbf{S}_{\Bbbk}$ the category such that
\begin{multline*}
\mathrm{ob}(\mathbf{S}_{\Bbbk})=\{X: X\ \mbox{is a}\ \Bbbk\mbox{-scheme and there exists a field extension}\ K\ \mbox{of}\ \Bbbk\\
 \mbox{such that}\ X\ \mbox{is isomorphic to a}\ K\mbox{-scheme of finite type}
\},
\end{multline*}
and for $X,Y\in \mathrm{ob}(\mathcal{S}_{\Bbbk})$,
\[
\Hom_{\mathbf{S}_{\Bbbk}}(X,Y)=\mbox{the set of morphisms from}\ X\ \mbox{to}\ Y\ \mbox{as}\ \Bbbk\mbox{-schemes}.
\]
 Note that fiber products in $\mathbf{S}_{\Bbbk}$ do not always exist, but a morphism $f:X\rightarrow Y$ in $\mathbf{S}_{\Bbbk}$ of finite type is \emph{quarrable}, i.e. the fiber product $X\times_Y Z$ exists in $\mathbf{S}_{\Bbbk}$ for any morphism $Z\rightarrow Y$ in $\mathbf{S}_{\Bbbk}$, and in fact it coincides with their fiber product in the category of all schemes.

 We recall the definition of rigid functors on $\mathbf{S}_{\Bbbk}$ introduced in \cite[definition 1.1]{Jan15}, with some slight modifications.
\begin{definition}\label{def-rigidfunctor}
A contravariant functor $V$ on $\mathbf{S}_{\Bbbk}$ with values in the category of abelian groups is called \emph{rigid} if it satisfies the following axioms:
\begin{enumerate}
	\item[(a)] For any finite flat morphism $\pi:X\rightarrow Y$ in $\mathbf{S}_{\Bbbk}$ there is a transfer morphism $\pi_*:V(X)\rightarrow V(Y)$ such that for another finite flat morphism $\rho:Y\rightarrow Z$, $(\rho\circ \pi)_*=\rho_*\circ \pi_*$. 
	\item[(b)] For a cartesian diagram of schemes in $\mathbf{S}_{\Bbbk}$
	\[
	\xymatrix{
	X' \ar[r]^{f'} \ar[d]_{\pi'} & X \ar[d]^{\pi} \\
	Y' \ar[r]^{f} & Y
	}
	\]
	with $\pi$ finite and flat, one has $f^*\pi_*=\pi'_{*}f'^{*}:V(X)\rightarrow V(Y')$.
	\item[(c)] If $X=X_1\coprod X_2$, the immersions $\pi_i:X_i\hookrightarrow X$ for $i=1,2$ induce an isomorphism
	\[
	(\pi_1^*,\pi_2^*):V(X)\xrightarrow{\sim} V(X_1)\oplus V(X_2)
	\]
	with inverse $\pi_{1*}+\pi_{2*}$.
	\item[(d)] If $X_m=X\times_{\mathbb{Z}}\Spec(\mathbb{Z}[T]/(T^m))$, then for the morphism $\pi:X_m\rightarrow X$ one has $\pi^*\pi_*=$ multiplication by $m$.
	\item[(e)] The projection $p:\mathbb{A}_X^1\rightarrow X$ induces an isomorphism $p^*:V(X)\xrightarrow{\sim} V(\mathbb{A}^1_X)$.
	\item[(f)] For a filtered projective system of \emph{regular} schemes $(X_i)_{i\in I}$ in $\mathbf{S}_{\Bbbk}$ with affine transition morphisms $X_i\rightarrow X_j$, the canonical map
	\begin{equation}\label{eq-axiom-f}
	\varinjlim_{i\in I} V(X_i)\rightarrow V(\varprojlim_{i\in I} X_i)
	\end{equation}
	is an isomorphism, provided that $\varprojlim_{i\in I} X_i$ exists in $\mathbf{S}_{\Bbbk}$ and is a regular scheme.
\end{enumerate}
\end{definition}

\begin{theorem}\label{thm-Jannsen}
Suppose $\Bbbk$ is an algebraically closed field. Let $K$ be an algebraically closed field containing $\Bbbk$, and $n$ an  integer. Let $V$ be a rigid functor on $\mathbf{S}_{\Bbbk}$. 
\begin{enumerate}
  \item[(i)] The canonical maps
\[
V(X)[n]\rightarrow V(X_K)[n], \quad V(X)/n\rightarrow V(X_K)/n
\]
are isomorphisms.
  \item[(ii)] Let $X$ be a scheme separated of finite type over $\Bbbk$. Then the canonical map $V(X)\rightarrow V(X_K)$ sends torsion free elements to torsion free elements.
\end{enumerate}
\end{theorem}
Proof: (i) is a special case of \cite[theorem 1.2]{Jan15}. For (ii), let $X$ be a scheme separated  of finite type over $L$, where $L$ is an extension field of $\Bbbk$, and let $\alpha$ be a torsion free element of $V(X)$. Suppose $\alpha$ maps to an $n$-torsion element in $V(X_K)$. By the axiom (f), there exists a scheme $Y$ separated of finite type over $\Bbbk$, such that $\alpha$ maps to an $n$-torsion element in $V(X\times_{\Bbbk}Y)$. But $Y$ has a $\Bbbk$-point, so $X\times_{\Bbbk}Y\rightarrow X$ has a section, which implies that $\alpha$ is $n$-torsion, a contradiction.\pqed

\begin{remark}
\begin{enumerate}
	\item[(i)] In \cite[definition 1.1]{Jan15} a rigid functor is defined on a category of schemes with flat morphisms. But the statement of \cite[theorem 1.2]{Jan15} and its proof need also at least pullbacks by regular embeddings of regular points, and the axiom (b) is applied in the case where $f$ and $f'$ are this sort of morphisms. So I believe a correct definition of a rigid functor is on a category admitting these morphisms, such as $\mathbf{S}_{\Bbbk}$.
	\item[(ii)] In the axiom (f) we assume that (\ref{eq-axiom-f}) holds only when the projective limit on the right-handside exists, i.e., we do not assume that this limit always exists. This is sufficient for the proof of \cite[theorem 1.2]{Jan15}. In fact, in that proof, this limit is always taken to be a field. 
	\item[(iii)] In the axiom (f) we assume that $X_i$ are regular for all $i\in I$. This is satisfied in  the proof of \cite[theorem 1.2]{Jan15} where axiom (f) is used.
	\item[(iv)] The assumption in \cite[theorem 1.2]{Jan15} that $V(Y)$ are $n$-torsion groups is redundant.
\end{enumerate}
\end{remark}

\subsection{Rigidity and mildness of Suslin homology}

Let $\Bbbk$ be a field. We will define a naive bivariant cohomology theory on $\mathbf{S}_{\Bbbk}$, for $Y\in \mathbf{S}_{\Bbbk}$, $X$ a scheme separated of finite type over $\Bbbk$:
\[
B_{r,i}(Y,X)=H^{-i}\big(\underline{C}_\bullet\big(c_{\mathrm{equi}}(X,r)\big)(Y)\big),
\]
where
\[
\underline{C}_i\big(c_{\mathrm{equi}}(X,r)\big)(Y)=c_{\mathrm{equi}}(X\times_{\Bbbk}Y\times_{\Bbbk}\Delta^{i}_{\Bbbk}/Y\times_{\Bbbk}\Delta^{i}_{\Bbbk},r).
\]

The following lemma follows immediately from the definition.
\begin{lemma}
Let $\Bbbk\hookrightarrow K$ be a field extension, and $\pi:Y\rightarrow \Spec(K)$ a $\Bbbk$-morphism of finite type. Then
 	\[
	B_{r,i}(Y,X)\cong H^{-i}\big(\underline{C}_\bullet
	\big(c_{\mathrm{equi}}(X_K,r)(Y)\big)\big),
	\]
  where
  \[
\underline{C}_i\big(c_{\mathrm{equi}}(X_K,r)\big)(Y)=c_{\mathrm{equi}}(X_K\times_{K}Y\times_{K}\Delta^{i}_{K}/Y\times_{K}\Delta^{i}_{K},r).
\]
	 In particular, for any scheme $X$ of finite type over $\Bbbk$,
\begin{equation}\label{eq-biv-suslin}
B_{0,i}(\Spec(K),X)\cong H_{i}^S(X_K,\mathbb{Z})
\end{equation}
where the right-handside is the Suslin homology of the $K$-scheme $X_K$.\pqed
\end{lemma}

\begin{proposition}\label{prop-rigid}
Let $\Bbbk$ be an algebraically closed field of characteristic 0.
The groups $B_{r,i}(Y,X)$ are contravariant  functorial in $Y$,  covariant functorial in $X$, and contravariant functorial in $X$ with respect to proper flat equidimensional morphisms. Moreover, for any $X$ of finite type over $\Bbbk$, $B_{0,i}(\ \cdot\ ,X)$ is a rigid functor on $\mathbf{S}_{\Bbbk}$.
\end{proposition}
Proof: By \cite[proposition 2.5.8]{Kel13}, $c_{\mathrm{equi}}(X,0)(-)$ is a \emph{presheaf with traces} (\cite[definition 3.3.1]{Kel13}). Thus $B_{0,i}$ satisfies the  properties (a)-(d) of definition \ref{def-rigidfunctor}. Properties (e) follows from \cite[lemma 4.1]{FV00}. To show the property (f), we need only to show  that for a filtered projective system of regular schemes $(Y_i)_{i\in I}$ in $\mathbf{S}_{\Bbbk}$ with affine transition morphisms $Y_i\rightarrow Y_j$ such that $Y=\varprojlim_{i\in I}Y_i$ is a regular scheme in $\mathbf{S}_{\Bbbk}$, we have isomorphisms for $n\geq 0$:
\begin{equation}\label{eq-prop-rigid-1}
\varinjlim_{i\in I}\underline{C}_n\big(c_{\mathrm{equi}}(X,0)\big)(Y_i)\xrightarrow{\sim}
\underline{C}_n\big(c_{\mathrm{equi}}(X,0)\big)(Y).
\end{equation}
By proposition \ref{prop-presheaves-cycles} (ii) and proposition \ref{prop-iso-symprod}, we have
\[
\underline{C}_n^{\mathrm{eff}}\big(c_{\mathrm{equi}}(X,0)\big)(Y_i)
\cong \mathrm{Hom}(Y_i\times_{\Bbbk}\Delta^n_{\Bbbk},\coprod_{j=0}^{\infty}\mathrm{S}^j(X)),
\]
and
\[
\underline{C}_n^{\mathrm{eff}}\big(c_{\mathrm{equi}}(X,0)\big)(Y)
\cong \mathrm{Hom}(Y\times_{\Bbbk}\Delta^n_{\Bbbk},\coprod_{j=0}^{\infty}\mathrm{S}^j(X)).
\]
Since $Y$ is noetherian, an element of $\mathrm{Hom}(Y\times_{\Bbbk}\Delta^n_{\Bbbk},\coprod_{j=0}^{\infty}\mathrm{S}^j(X))$ in fact  lies in $\mathrm{Hom}(Y\times_{\Bbbk}\Delta^n_{\Bbbk},\coprod_{j=0}^{m}\mathrm{S}^j(X))$ for some $m>0$. But $\coprod_{j=0}^{m}\mathrm{S}^j(X)$ is an algebraic space of finite type over $\Bbbk$, so 
\[
\varinjlim_{i\in I}\mathrm{Hom}(Y_i\times_{\Bbbk}\Delta^n_{\Bbbk},\coprod_{j=0}^{m}\mathrm{S}^j(X))\cong \mathrm{Hom}(Y\times_{\Bbbk}\Delta^n_{\Bbbk},\coprod_{j=0}^{m}\mathrm{S}^j(X)),
\]
and (\ref{eq-prop-rigid-1}) follows.
 \pqed

So  we can apply theorem \ref{thm-Jannsen} to the functor $B_{0,i}(-,X)$, and by (\ref{eq-biv-suslin}) we obtain:
\begin{corollary}
Let $K$ and $L$ be algebraically closed fields of characteristic 0, and $\iota:K\rightarrow L$ an imbedding. Let $X$ be scheme separated and  of finite type over $K$. Then the natural map $\varphi: H_i^S(X,\mathbb{Z})\rightarrow H_i^S(X_L,\mathbb{Z})$ induces isomorphisms
\[
H_i^S(X,\mathbb{Z})[n]\rightarrow H_i^S(X_L,\mathbb{Z})[n],\ 
H_i^S(X,\mathbb{Z})/n\rightarrow H_i^S(X_L,\mathbb{Z})/n,
\]
and for any torsion free element $\alpha\in H_i^S(X,\mathbb{Z})$, $\varphi(\alpha)$ is torsion free.\pqed
\end{corollary}

Now we are ready to show the main theorems of this paper.
\begin{theorem}\label{thm-mild-2}
Let $\Bbbk$ be an algebraically closed field of characteristic 0, $X$ a separated scheme of finite type over $\Bbbk$. Then $H_i^S(X,\mathbb{Z})$ is mild, and
\begin{enumerate}
	\item[(i)] $s(H_{i-1}^S(X,\mathbb{Z}))+r(H_{i}^S(X,\mathbb{Z}))=r_{\mathbb{Q}_l}(H_{\et}^i(X,\mathbb{Q}_{\ell}))$, where $r_{\mathbb{Q}_{\ell}}$ means the rank of a $\mathbb{Q}_{\ell}$-vector space.
	\item[(ii)] There are canonical isomorphisms
	\[
		H_i^S(X,\mathbb{Z})_{\mathrm{tor}}\cong \bigoplus_{\ell}H^{i+1}_{\et}(X,\mathbb{Z}_{\ell})_{\mathrm{tor}}.
	\]
\end{enumerate}
\end{theorem}
Proof: Fix $X$, a  separated scheme  of finite type over $\Bbbk$.
Let $J$ be the category such that the set of objects consists of all the pairs $(Y\rightarrow \Spec(K), f:Y\times_K \Bbbk\rightarrow X)$, where $K$ is an algebraically closed subfield of $\Bbbk$ with $\mathrm{trdeg(K/\mathbb{Q})}<+\infty$, $Y\rightarrow \Spec(K)$ a separated morphism of finite type, and $f$ an isomorphism of $\Bbbk$-schemes, and 
 	\begin{eqnarray*}
 	&&\Hom_J\big((Y_1\rightarrow \Spec(K_1), f_1:Y_1\times_{K_1} \Bbbk\rightarrow X),(Y_2\rightarrow \Spec(K_2), f_2:Y_2\times_{K_2} \Bbbk\rightarrow X)\big)\\
 	&=&\begin{cases}
 	\emptyset, & \mbox{if}\ K_1\not\subset K_2,\\
 	\{g: Y_1\times_{K_1}K_2\xrightarrow{\sim} Y_2| f_2\circ(g\times_{K_2}\Bbbk)=f_1\}, & 
 	\mbox{if}\ K_1\subset K_2.
 	\end{cases}
 	\end{eqnarray*}
 By the Lefschetz principle, the category $J$ is filtered. 

 A morphism $g$ as above induces a homomorphism
 \[
 H_i^S(Y_1,\mathbb{Z})\rightarrow H_i^S(Y_2,\mathbb{Z}),
 \]
 where the left-handside is the Suslin homology of the $K_1$-scheme $Y_1$, and the right-handside is that of the $K_2$-scheme $Y_2$. By the Lefschetz principle again we have a natural isomorphism
 \begin{equation}
 \varinjlim_{J}H_i^S(Y_j,\mathbb{Z})\xrightarrow{\sim}
 H_i^S(X,\mathbb{Z}).
 \end{equation}
 By theorem \ref{thm-mild-1}, applying proposition \ref{prop-limit}, we are done.\pqed
 
\begin{theorem}\label{thm-mild-3}
Let $\Bbbk$ be an algebraically closed field of characteristic 0, $X,Y$  separated schemes of finite type over $\Bbbk$. 
\begin{enumerate}
  \item[(i)] For a morphism of $\Bbbk$-schemes $f:X\rightarrow Y$, the induced homomorphism
  $f_*:H_i^S(X,\mathbb{Z})\rightarrow H_i^S(Y,\mathbb{Z})$ of mild abelian groups satisfies the condition (M1).
    \item[(ii)] For a morphism of $\Bbbk$-schemes $f:X\rightarrow X$, the induced endomorphism
  $f_*:H_i^S(X,\mathbb{Z})\rightarrow H_i^S(X,\mathbb{Z})$ of mild abelian groups satisfies the condition (M2).
\end{enumerate}
\end{theorem}
Proof: (i) follows from theorem \ref{thm-mild-1} (iii) and proposition \ref{prop-limit-hom}, and (ii) follows from theorem \ref{thm-mild-1} (iv) and proposition \ref{prop-limit-endo}.\pqed

\begin{remark}
As shown by Geisser \cite{Gei17}, for smooth schemes $X$, $r(H_i^S(X,\mathbb{Z})=0$ for $i>0$.  In general for a singular scheme $X$, $r(H_i^S(X,\mathbb{Z})$ may not be zero for $i>0$. To see this, consider a singular proper curve $C$ over $\mathbb{C}$. Let $\pi:\tilde{C}$ be the normalization of $C$. Let $Z=\{z_1,\cdots,z_k\}$ be the set of singular points of $C$, and $\pi^{-1}(Z)=\{y_1,\cdots,y_m\}$. Then one has the abstract blowup exact sequence (\cite[\S 9]{FV00})
\begin{eqnarray*}
H_1^S(\pi^{-1}(Z),\mathbb{Z})\rightarrow H_1^S(\tilde{C},\mathbb{Z})\oplus H_1^S(Z,\mathbb{Z})\rightarrow
H_1^S(C,\mathbb{Z})\rightarrow\\
\rightarrow H_0^S(\pi^{-1}(Z),\mathbb{Z})\rightarrow H_0^S(\tilde{C},\mathbb{Z})\oplus H_0^S(Z,\mathbb{Z})\rightarrow
H_0^S(C,\mathbb{Z})\rightarrow 0.
\end{eqnarray*}
It is easily seen that $H_1^S(\pi^{-1}(Z),\mathbb{Z})=H_1^S(Z,\mathbb{Z})=0$, $r(H_1^S(\tilde{C},\mathbb{Z}))=0$, and $H_0^S(\pi^{-1}(Z),\mathbb{Z})=\mathbb{Z}^{m}$. Thus $r(H_1^S(C,\mathbb{Z}))=0$ if and only if the homomorphism
\begin{eqnarray*}
H_0^S(\pi^{-1}(Z),\mathbb{Z})\rightarrow H_0^S(\tilde{C},\mathbb{Z})\oplus H_0^S(Z,\mathbb{Z})
\end{eqnarray*}
is injective. But one can find a smooth proper curve $\tilde{C}$ and distinguished points $y_1,\cdots,y_m$ on $\tilde{C}$ such that $y_1+\cdots+y_l-y_{l+1}-\cdots-y_m$ is zero in $H_0^S(\tilde{C},\mathbb{Z})=CH^1(\tilde{C})$ for some $0\leq l\leq m$. Gluing $y_1,\cdots,y_m$ we obtain a curve $C$ with a normal crossing singular point $z$, such that the above homomorphism is not injective.

\end{remark}

\section{Further problems}
In this section we make two conjectures on the generalizations of theorems \ref{thm-mild-2} and \ref{thm-mild-3}. 
\subsection{Correspondence homomorphisms of Suslin homology}
Let $\Bbbk$ be an algebraically closed field of characteristic 0, and $X,Y\in \mathbf{Sch}/\Bbbk$. Let $X\times_{\Bbbk} Y\rightarrow X\xrightarrow{p} \Spec(\Bbbk)$ be the obvious projections. By \cite[3.6.3, 3.7.5]{SV00}, there is a \emph{correspondence homomorphism} of presheaves
\[
p_*(c_{\mathrm{equi}}(X\times_{\Bbbk} Y/X,0))\otimes c_{\mathrm{equi}}(X/\Spec(\Bbbk),0)\rightarrow c_{\mathrm{equi}}(X\times_{\Bbbk} Y/\Spec(\Bbbk),0)
\]
and a pushforward homomorphism of presheaves
\[
c_{\mathrm{equi}}(X\times_{\Bbbk} Y/\Spec(\Bbbk),0)\rightarrow c_{\mathrm{equi}}(Y/\Spec(\Bbbk),0).
\]
Since $p_*(c_{\mathrm{equi}}(X\times_{\Bbbk} Y/X,0))(\Delta^i)=c_{\mathrm{equi}}(X\times_{\Bbbk} Y\times_{\Bbbk}\Delta^i/X\times_{\Bbbk}\Delta^i,0))$, every face map $\Delta^{i-1}\hookrightarrow \Delta^i$ induces a commutative diagram
\begin{tiny}
\begin{equation*}
\xymatrix{
  c_{\mathrm{equi}}(X\times Y\times\Delta^i/X\otimes\Delta^i,0)
  \otimes c_{\mathrm{equi}}(X\times \Delta^i/\Delta^i,0) \ar[r] \ar[d] &
  c_{\mathrm{equi}}(X\times Y\times \Delta^i/\Delta^i,0) \ar[r] \ar[d] & 
  c_{\mathrm{equi}}(Y\times \Delta^i/\Delta^i,0)\ar[d] \\
  c_{\mathrm{equi}}(X\times Y\times\Delta^{i-1}/X\otimes\Delta^{i-1},0)
  \otimes c_{\mathrm{equi}}(X\times \Delta^{i-1}/\Delta^{i-1},0) \ar[r]  &
  c_{\mathrm{equi}}(X\times Y\times \Delta^{i-1}/\Delta^{i-1},0) \ar[r]  & 
  c_{\mathrm{equi}}(Y\times \Delta^{i-1}/\Delta^{i-1},0)
}
\end{equation*}
\end{tiny}
which induce a correspondence homomorphism
\begin{equation}\label{eq-corresp-suslin}
c_{\mathrm{equi}}(X\times Y/X,0)\otimes H_i^S(X,\mathbb{Z})\rightarrow H_i^S(Y,\mathbb{Z}).
\end{equation}
For $\alpha\in c_{\mathrm{equi}}(X\times Y/X,0)$, we denote by $\mathrm{tr}_{\alpha}$ the homomorphism $H_i^S(X,\mathbb{Z})\rightarrow H_i^S(Y,\mathbb{Z})$ induced by $\alpha$.
The following conjecture\footnote{This conjecture has been confirmed by Bruno Kahn; see the appendix.} is to generalize theorem \ref{thm-mild-3} to the homomorphisms induced by any correspondence, not necessarily the ones induced by morphisms.
\begin{conjecture}\phantomsection\label{conj-cor-hom}
\begin{enumerate}
  \item[(i)] For any $\alpha\in c_{\mathrm{equi}}(X\times Y/X,0)$, $\mathrm{tr}_{\alpha}$ satisfies (M1).
  \item[(ii)] For any $\alpha\in c_{\mathrm{equi}}(X\times X/X,0)$, $\mathrm{tr}_{\alpha}$ satisfies (M2).
\end{enumerate}

\end{conjecture}

\subsection{In positive characteristic}
Let $\Bbbk$ be an algebraically closed field of positive characteristic $p$. 
\begin{conjecture}\label{conj-pmild}
Let $X$ be a scheme separated of finite type over $\Bbbk$. Then
\[
H_i^S(X,\mathbb{Z}[\frac{1}{p}])\cong V\oplus \bigoplus_{\ell\neq p}(\mathbb{Q}_p/\mathbb{Z}_p)^{s_{i,\ell}}\oplus \mathbb{Z}[\frac{1}{p}]^{\oplus r_i} \oplus T
\]
where $V$ is a uniquely divisible group, and $T=\bigoplus_{\ell\neq p}H^{i+1}_{\et}(X,\mathbb{Z}_{\ell})_{\mathrm{tor}}$, and 
\[
s_{i-1,\ell}+r_i=b_{i,\ell}(X)=\mathrm{rank}(H^{i}_{\et}(X,\mathbb{Z}_{\ell})).
\]
\end{conjecture}
This conjecture implies that $H_S^i(X,\mathbb{Z}/n)\cong H_{\et}^i(X,\mathbb{Z}/n)$, which is shown in \cite{SV96}.
Furthermore we conjecture that $s_{i,\ell}$ is independent of $\ell$. It is interesting to ask whether theorem \ref{thm-mild-3} and even conjecture \ref{conj-cor-hom} hold in positive characteristic. In general there is no simultaneous choice of $W$ in definition \ref{def-cond-M2} such that all the endomorphisms $f$ induced by self-morphisms of $X$ lift to $\tilde{f}$, by Serre's counterexample \cite[exposé IX]{SGA4}.

\vspace{0.5cm}
\emph{Acknowledgement}:  I am grateful to Peng Sun and Jinxing Xu for helpful discussions and their encouragement to complete this paper.  I thank Yong Hu,  Mao Sheng, Jilong Tong and Fei Xu, for I  benefited a lot from the reading seminars on algebraic cycles they organized. Thanks also Spencer Bloch, Zhen Huan, Junchao Shentu, Jiangwei Xue, Sen Yang,  Jiu-Kang Yu, and Lei Zhang for related discussions. Finally I thank Bruno Kahn for telling me his proof and  for writing the appendix.

This work is supported by CPSF 2015T80007 and NSFC 34000-31610265, 34000-41030364, 34000-41030338.\\

\appendix

\section{Appendix by B. Kahn}


In this appendix, we use Voevodsky's triangulated categories of motives over $k$ and (when $k=\C$) the Betti realization functor to simplify the proofs of Theorems \ref{thm-mild-1}, \ref{thm-mild-2} and \ref{thm-mild-3} and generalize them, for example answering Conjecture \ref{conj-cor-hom} positively. We comment on Conjecture \ref{conj-pmild} in Remark \ref{r1}.


We first note that an abelian group $G$ is mild if and only if there exists a $\Q$-vector space $V$ and a homomorphism $f:V\to G$ such that $\Ker f$ and $\Coker f$ are finitely generated. 
Let us call such an $f$ a \emph{presentation}. Then presentations naturally form a category $\sC$ with a forgetful functor $F:\sC\to \Ab$ (abelian groups), $f\mapsto G$, which sends presentations to mild groups and morphisms to those verifying Conditions (M1) and (M2). 

The following features of $\sC$ are not used in the sequel, but  show that it is a good object philosophically:

\begin{itemize}
\item $\sC$ is an abelian category.
\item $F$ is exact and conservative. (This gives a kind of uniqueness to a presentation of a mild group $G$.)
\item The homology functor $\sC\to \Ab\times \Ab$ is conservative: if $\phi:f\to g$ is a morphism in $\sC$ and $H_0(\phi),H_1(\phi)$ are isomorphisms, then $\phi$ is an isomorphism. 
\end{itemize}

The strategy is now to associate to Suslin homology groups $H_i^S(X)$, when $k=\C$, a presentation which is (weakly) functorial for morphisms $X\to Y$ and more generally for correspondences.

For this, we start from Voevodsky's ``big'' category of \'etale motives $\DM^\et(\C)\allowbreak=:\DM^\et$. There is a ($\otimes$-triangulated)  ``Betti'' realization functor
\[B^*:\DM^\et\to D(\Ab)\]
where the right hand side is the derived category of abelian groups. It can be obtained in two steps. First, Ayoub constructs such a realization functor in  \cite[\S 2]{ayoub}, but starting from the category $\DA^\et$ of ``motives without transfers''; then, the natural $\otimes$-functor $\DA^\et\to \DM^\et$ is an equivalence of $\otimes$-triangulated categories by Cisinski-D\'eglise (\cite[Ann. B]{realetale}, \cite[Cor. 5.5.5 and 5.5.7]{CD}).

The category $\DM^\et$ is compactly generated and  $B^*$ commutes with infinite direct sums, hence has a right adjoint $B_*$ by Neeman's Brown representability theorem \cite[Th. 4.1]{neeman}. For any $M\in \DM^\et$, we have the unit map of this adjunction:
\[\eta_M:M\to B_*B^*(M).\]

Write $\DM^\et_\gm$ for the thick triangulated subcategory of $\DM^\et$ generated by motives of the form $M(X)(r)$ with $X$ smooth and $r\in \Z$: this is a rigid $\otimes$-triangulated subcategory.

\begin{lemma} Let $M\in \DM^\et_\gm$. Then $\eta_{M\otimes \Z/n}$ is an isomorphism for any $n>0$.
\end{lemma}

\begin{proof} By Yoneda, we have to prove that for any $N\in \DM^\et_\gm$, the morphism
\begin{multline*}
\DM^\et_\gm(N,M\otimes \Z/n)\to \DM^\et_\gm(N,B_*B^*(M\otimes \Z/n))\\
\simeq D(\Ab)(B^*(N),B^*(M\otimes \Z/n))
\simeq D(\Ab)(B^*(N),B^*(M)\otimes \Z/n )
\end{multline*}
is an isomorphism. By [symmetric monoidal] rigidity ($\DM^\et_\gm(N,M\otimes \Z/n)=\DM^\et_\gm(N\otimes M^*, \Z/n)$, etc.), we may assume $M=\Z$. The statement is then stable under taking cones and direct summands. Thus we may reduce to generators $N=M(X)(-r)[-i]$, where $X$ is smooth and $i\in \Z$, $r\ge 0$. In this case, by the cancellation theorem of \cite{voecan} (see \cite[Prop. A.3]{HK} for its \'etale version), the statement becomes that the map
\[H^i_{M,\et}(X,\Z/n(r))\to H^i_B(X,\Z/n(r))=H^i_B(X,\Z/n)\]
is an isomorphism, where the left (resp. right) hand side is \'etale motivic (resp. Betti) cohomology. This follows from \cite[Th. 10.2]{MVW} and the comparison theorem between Betti and \'etale cohomology with finite coefficients \cite{artin}.
\end{proof}

We now choose a cone $C(M)$ of $\eta_M$, for all $M\in \DM^\et_\gm$. This cone is not unique or functorial, but it is weakly so by axiom TR3 of triangulated categories: given $M'$ and a choice $C(M')$, for any morphism $\phi:M\to M'$ one can find a morphism $C(\phi):C(M)\to C(M')$ which makes the corresponding diagram of triangles commutative. This will be sufficient for us.

By the lemma, multiplication by $n$ is an isomorphism on $C(M)$ for all $n\ne 0$ (use the ``triangulated nine lemma'' of \cite[Prop. 1.11.1]{BDD}). On the other hand, $D(\Ab)(\Z[i],B^*(M))$ is finitely generated for all $i\in \Z$ (reduce to $M=M(X)(r)$ as in the proof of the lemma); hence the map
\[\DM^\et(\Z[i+1],C(M))\to \DM^\et(\Z[i],M)\]
is a presentation of the right hand side for all $i\in \Z$, and a morphism $\phi:M\to M'$ yields a morphism of presentations. 

If $M=M(X)$ for $X$ separated of finite type, then $\DM^\et(\Z[i],M)\simeq H_i^S(X)$; this concludes the proof in the case $k=\C$, generalizing the statements to all objects and morphisms of $\DM^\et_\gm$. This then extends to all Hom groups of $\DM_\gm^\et$ by rigidity. The assertions on rank and corank are also easily recovered. 
In the general case, the Suslin rigidity argument of \S 4 can be more conceptually replaced by \cite[Prop. 3.3.3]{voetri}.

\begin{rk}\label{r1} If $k$ has positive characteristic $p$, we may use the same method with the $l$-adic realization functors of \cite{realetale}, for weaker results. Namely, for a prime number $l$, say that an abelian group $G$ is $l$-mild if there exists a $\Z_l$-linear homomorphism $f:V\to G\otimes \Z_l$, with $V$ a $\Q_l$-vector space, where $\Ker f$, $\Coker f$ are finitely generated $\Z_l$-modules. 
Then  for all $l\ne p$, $\DM^\et(N,M)$ is $l$-mild for any $N,M\in \DM^\et_\gm$
. (To apply this to Suslin homology, use \cite{Ke}.)

Unfortunately, this notion is hard to globalize: the group $G=\bigoplus_l \Z/l$ is $l$-mild for all $l$, but is not mild. Similarly, independence of $l$ for the corank of $H_i^S(X)\{l\}$ is essentially equivalent to the conjectures on independence of $l$ for the ranks of $l$-adic (co)homology groups.

On the other hand, \cite[Th. A.1.3]{ABB} shows that the finitely generated quotient of the \'etale motivic cohomology of smooth $k$-varieties given by the above approach is finite in some range; this relies on \cite{WeilII} (an idea which goes back to S. Bloch). Thus, arithmetic yields extra information on the integers $r(G)$ and $s(G)$, where $G$ is a Hom group in $\DM_\gm^\et$.
\end{rk}

\enlargethispage*{20pt}

\textsc{School of Mathematics, Sun Yat-sen University, Guangzhou 510275, P.R. China}

 \emph{E-mail address:}  huxw06@gmail.com\\

\textsc{IMJ-PRG Case 247 4 place Jussieu 75252 Paris Cedex 05 France}

\emph{E-mail address:} bruno.kahn@imj-prg.fr

\end{document}